\documentclass[11pt,reqno]{amsart}
\usepackage{color,mathrsfs,epsfig}

\let\dsp=\displaystyle
\newcommand{\N}{\mathbb{N}}
\newcommand{\R}{\mathbb{R}}
\newcommand{\eps}{\varepsilon}
\newcommand{\Leb}{{\mathscr L}}
\newcommand{\ind}{{\bf 1}\,}
\renewcommand{\div}{{\rm div}\,}
\newcommand{\curl}{{\rm curl}\,}
\newcommand{\D}{\mathcal{D}}
\newcommand{\loc}{{\rm loc}}
\def\mint
    {\mkern12mu\hbox{\vrule height4pt
    depth-3.2pt width5pt}%
    \mkern-16.5mu\int}

\newcommand{\spt}{{\rm spt}\,}
\newcommand{\lpn}{\left|\left|\left|}
\newcommand{\rpn}{\right|\right|\right|}
\renewcommand{\S}{\mathscr{S}}

\newtheorem{theorem}{Theorem}[section]
\newtheorem{lemma}[theorem]{Lemma}
\newtheorem{propos}[theorem]{Proposition}
\newtheorem{corol}[theorem]{Corollary}

\theoremstyle{definition}
\newtheorem{definition}[theorem]{Definition}

\theoremstyle{remark}
\newtheorem{remark}[theorem]{Remark}

\numberwithin{equation}{section}

\begin{document}
\title[Vector fields with gradient given by a singular integral]{Lagrangian flows for vector fields with 
gradient \\ given by a singular integral}

\author{Fran\c{c}ois Bouchut}
\address{Fran\c{c}ois Bouchut, LAMA, CNRS \& Universit\'e Paris-Est-Marne-la-Vall\'ee, 
5 boulevard Descartes, Cit\'e Descartes - Champs-sur-Marne,
77454 Marne-la-Vall\'ee cedex 2, France}
\email{francois.bouchut@univ-mlv.fr} 
\author{Gianluca Crippa}
\address{Gianluca Crippa, Departement Mathematik und Informatik, Universit\"at Basel, 
Rheinsprung 21, CH-4051, Basel, Switzerland}
\email{gianluca.crippa@unibas.ch}

\begin{abstract} We prove quantitative estimates on flows of ordinary differential equations with vector field with gradient given by a singular integral of an $L^1$ function. Such estimates allow to prove existence, uniqueness, quantitative stability and compactness for the flow, going beyond the $BV$ theory. We illustrate the related well-posedness theory of Lagrangian solutions to the continuity and transport equations.
\end{abstract}

\maketitle

\section{Introduction}

\subsection{Ordinary differential equations with non smooth vector field}

When $b : [0,T] \times \R^N \to \R^N$ is a bounded smooth vector field, the flow of $b$ is the smooth map $X : [0,T] \times \R^N \to \R^N$ satisfying
\begin{equation}\label{e:introode}
\begin{cases}
\displaystyle \frac{dX}{ds}(s,x) = b \big( s,X(s,x)\big) \,, \quad s\in[0,T], \\ \\
X(0,x) = x \,.
\end{cases}
\end{equation}
The possibility of going beyond the smooth framework \eqref{e:introode} has been studied in recent years by several authors. In the non smooth context, a convenient notion of generalized flow is that of regular Lagrangian flow. Roughly speaking, it amounts to requiring that \eqref{e:introode} is satisfied (in weak sense) for almost every $x \in \R^N$, and that for any time $s\in [0,T]$ the map $X(s,\cdot) : \R^N \to \R^N$ is almost preserving (lower bounds on) the Lebesgue measure of sets (see Definition~\ref{d:rlf} for the precise conditions).

Existence, uniqueness and stability of regular Lagrangian flows have been first proved by DiPerna and Lions \cite{DL} for Sobolev vector fields with bounded divergence. Such result was later extended by Ambrosio \cite{A} to $BV$ vector fields with bounded divergence. 
The argument in both proofs is quite indirect and exploits the connection between \eqref{e:introode} and the Cauchy problem for the continuity equation
\begin{equation}\label{e:intropde}
\partial_t u(t,x) + \div \big( b(t,x) u(t,x) \big) = 0\,,
\end{equation}
together with the theory of renormalized solutions for \eqref{e:intropde}.
In this approach, an important technical tool is the regularization by a smooth kernel and its commutator with the transport operator.
For a detailed account on these results, we suggest for instance \cite{notes,notes2,thesis,D}. Further results can be found in \cite{ALM,ADM,BC,hyp}.

\subsection{Quantitative estimates for Sobolev vector fields}\label{ss:sob}

In \cite{estimates} it was shown that many of the ODE results of the DiPerna-Lions theory can be recovered with simple a priori estimates and functional inequalities, directly in Lagrangian formulation, that is without exploiting the connection with the continuity equation~\eqref{e:intropde}. 

The basic idea of \cite{estimates} is to consider an integral functional measuring the distance between two eventual regular Lagrangian flows. If $X$ and $\bar X$ are regular Lagrangian flows associated to the same vector field $b$, given $\delta >0$ we introduce the quantity
\begin{equation}
	\Phi_\delta (s) = \int \log \left( 1 + \frac{|X(s,x) - \bar X(s,x)|}{\delta} \right) \, dx
	\label{e:Phidelta}
\end{equation}
(suitable truncations or localizations are necessary to make this integral convergent, but for simplicity we skip such technical point in this introductory presentation).

If uniqueness fails, it is possible to find a time $s$ and a set $A \subset \R^N$ with $\Leb^N(A) = \alpha >0$ such that $|X(s,x) - \bar X(s,x)| \geq \gamma >0$ for $x\in A$, implying the lower bound
\begin{equation}\label{e:condition}
\Phi_\delta (s) \geq \int_A \log \left( 1 + \frac{\gamma}{\delta} \right) \, dx = 
\alpha \log \left( 1 + \frac{\gamma}{\delta} \right) \,.
\end{equation}
However, time differentiation of $\Phi_\delta$ gives
\begin{equation}\label{e:differentiate}
\Phi'_\delta (s) \leq \int \frac{|b(X) - b (\bar X)|}{\delta + |X - \bar X|} \, dx \leq
\int \min \left\{ \frac{2 \|b\|_\infty}{\delta} \; ; \; \frac{|b(X) - b(\bar X)|}{|X - \bar X|} \right\} \, dx \,.
\end{equation}
The key remark of \cite{estimates} is that the estimate of the difference quotients in terms of the maximal function
\begin{equation}\label{e:quotients}
\frac{|b(x)-b(y)|}{|x-y|} \leq C \big( M Db(x) + M Db(y) \big) 
\quad \text{ for a.e.~$x, y \in \R^N$}
\end{equation}
allows to conclude from \eqref{e:differentiate} (in which we simply drop the first element in the minimum) that
\begin{equation}\label{e:tobebound}
\Phi'_\delta (s) \leq C \int \big( M Db (X) + M Db (\bar X) \big) \, dx \leq
C \int M Db \, dx \,,
\end{equation}
changing variable along the flow. Recalling that the maximal function enjoys the strong estimate 
\begin{equation}\label{e:strong}
\| M Db \|_{L^p} \leq C \| Db \|_{L^p} \quad \text{for any $p>1$,}
\end{equation}
we see that in the case $b \in W^{1,p}$ with $p>1$ the estimate \eqref{e:tobebound} gives an upper bound on $\Phi'_\delta$ (and hence on $\Phi_\delta$) uniformly with respect to $\delta$. But this is in contrast with the non uniqueness lower bound \eqref{e:condition} when $\delta\rightarrow 0$, hence we obtain uniqueness of the regular Lagrangian flow for vector fields in $W^{1,p}$ with $p>1$.

In the same fashion, estimates providing stability (with quantitative rates), compactness and some mild regularity of the flow have been obtained in \cite{estimates}. 

\subsection{The case when the gradient is a singular integral: sketch of the proof of uniqueness}\label{ss:sketch}

The proof of \cite{estimates} breaks down when $p=1$, since the strong estimate \eqref{e:strong} does not hold any more and only the weak estimate
\begin{equation}\label{e:weak}
\lpn M Db \rpn_{M^1} :=
\sup_{\lambda>0} \Big\{ \lambda \, \Leb^N \big( \big\{ x \in \R^N \; : \; |M Db (x)| > \lambda \big\} \big) \Big\}
\leq C \| Db \|_{L^1}
\end{equation}
is available (see Section~\ref{s:back} for a description of the space $M^1$). In general $M Db$ does not even belong to $L^1_\loc$, hence the integral in \eqref{e:tobebound} is not finite any more. This was the obstruction for the strategy of \cite{estimates} to reach the cases $W^{1,1}$ and $BV$.

\medskip

In the present paper we show how it is possible to modify the proof of \cite{estimates} in order to get the $W^{1,1}$ case, and indeed the same idea allows the treatment of vector fields whose derivative can be expressed as a singular integral of an $L^1$ function, that is
\begin{equation}
	Db = K * g \,,\quad g\in L^1,
	\label{e:Db=SL1}
\end{equation}
where the singular kernel $K$ is smooth out of the origin, grows at most as $|x|^{-N}$ in $\R^N$, and satisfies suitable cancellation properties (see Section~\ref{s:back} for the definition of singular integrals, and Section~\ref{s:estimate} for the precise description of the class of vector fields we are considering). Here we just notice that this class is natural in the context of the study of some nonlinear PDEs (see Section~\ref{ss:extapp} for some detail), and that such class is not contained in $BV$, neither contains it.
Nevertheless, \eqref{e:Db=SL1} contains the regularity $W^{1,1}$ (take $K=\delta_0$), already considered in \cite{Jab} (see also \cite{ChJ}).
Our main results were announced in \cite{BCX}.

\medskip

We now informally describe the key steps of our proof. Remember that in the $L^1$ context singular integrals do not enjoy strong estimates, but only the weak estimate analogue to \eqref{e:weak}. When looking at the estimate for the difference quotients \eqref{e:quotients}, we realize that the composition of the two operators (maximal function and singular integral) is involved: the quantity $M ( K*g)$ appears on the right hand side. In order to preserve in this composition the weak estimate (which holds separately for the maximal function and for the singular integral), it is convenient to consider a smooth variant of the usual maximal function: given $\rho \in C^\infty_c(\R^N)$ we set
\begin{equation}
	M_\rho (u) (x) = \sup_{\varepsilon>0} \left| \frac{1}{\varepsilon^N} \int_{\R^N} \rho\left( \frac{x-y}{\varepsilon} \right) u(y) \, dy \right| \,.
	\label{e:Mrho}
\end{equation}
We are now taking smooth averages, and we are taking the absolute value after having computed the average. This allows cancellations which play together with the cancellations of the singular kernel~$K$. With this definition, we can deduce that the composition satisfies the weak estimate
\begin{equation}\label{e:weak2}
\lpn M_\rho (K * g) \rpn_{M^1} := 
\sup_{\lambda>0} \Big\{ \lambda \, \Leb^N \big( \big\{ x \in \R^N 
\; : \; |M_\rho (K * g) (x)| > \lambda \big\} \big) \Big\}
\leq C\| g \|_{L^1}
\end{equation}
(which cannot be obtained simply by composing the two weak estimates), and we still have the analogue of \eqref{e:quotients}, which now reads for some $\rho$
\begin{equation}\label{e:quotients2}
\frac{|b(x)-b(y)|}{|x-y|} \leq C \big( M_\rho (K * g) (x) + M_\rho (K * g) (y) \big) 
\quad \text{ for a.e.~$x, y \in \R^N$.}
\end{equation}
Going back to \eqref{e:differentiate}, we see that now we have to estimate the integral of the minimum of two functions: the first one is $L^\infty$, but with a norm which grows when $\delta\rightarrow 0$, while the $M^1$ (pseudo)norm of the second one is bounded by $\|g\|_{L^1}$. None of the bounds by itself is sufficient, but an interpolation inequality allows to conclude
\begin{equation}\label{e:critical}
\Phi'_\delta (s) \leq C \|g\|_{L^1} \left[ 1 + \log \left( \frac{C}{\delta \|g\|_{L^1}} \right) \right] \,.
\end{equation}
Recalling \eqref{e:condition}, we discover that we are exactly on the critical rate for the uniqueness: both the lower and the upper bounds behave like $\log (1/\delta)$ for $\delta$ small. By itself, \eqref{e:critical} is not sufficient to imply uniqueness. But $g\in L^1$ can be decomposed as $g = g^1 + g^2$, where the $L^1$ norm of $g^1$ is as small as desired, and $g^2 \in L^2$. We apply to $g^2$ the arguments of the $W^{1,p}$ case (as in Section~\ref{ss:sob}), and we are left with $g^1$ instead of $g$ in \eqref{e:critical}: choosing its $L^1$ norm small enough, we obtain a contradiction with \eqref{e:condition}, deducing uniqueness also in this case.

\subsection{Consequences, applications and possible extensions}\label{ss:extapp}

The approach we have just presented allows a complete theory of regular Lagrangian flows for vector fields whose gradient is given by a sum of singular integrals of $L^1$ functions: existence, uniqueness, quantitative stability and compactness can be proved. This gives a well-posed notion of flow, which satisfies the usual semigroup property.

\medskip

Being our derivation performed at ODE level, no direct consequences are available for distributional solutions of the continuity equation \eqref{e:intropde} (or for the closely related transport equation). However, the ODE well-posedness brings as consequence the well-posedness for Lagrangian solutions to such PDEs, that is, for solutions which are naturally associated to flows. This enables us to construct a unique stable semigroup of PDE solutions.

\medskip

The class of vector fields we are considering is natural in view of some applications to nonlinear PDEs. As an example, the two-dimensional Euler equation in vorticity form reads
$$
\partial_t \omega + \div (v \, \omega) = 0 \,,
$$
where the vorticity $\omega$ is the rotational of the velocity $v$, that is, $\omega = \curl v$. This can be equivalently rewritten through the Biot-Savart law as
$$
v(t,x) = \frac{1}{2\pi} \int_{\R^2} \frac{(x-y)^\perp}{|x-y|^2} \, \omega (t,y) \, dy \,.
$$
If we look at vorticities $\omega \in L^\infty_t (L^1_x)$, the velocity $v$ is precisely in the setting under consideration in this paper. A standard smoothing procedure, together with the compactness of the associated ODE flows, imply existence of Lagrangian solutions for the Euler equation. This will be contained in the follow-up paper \cite{BC2}, together with similar applications to the Vlasov-Poisson equation. 

\medskip

A relevant case to be understood is that of singular integrals of measures, rather than of $L^1$ functions. Looking back at the streamline of the proof in Section~\ref{ss:sketch}, we see that we can carry on our analysis with $g$ being a measure until \eqref{e:critical}. Integrability of $g$ is just needed in the subsequent decomposition $g = g^1 + g^2$, in order to gain the required smallness. A full extension of our proof to the case of a measure would imply a proof of Bressan's compactness conjecture \cite{BRc,BRm}, together with other deep consequences.
There are however limitations to such theory, because of counterexamples given by Bressan \cite{BRc} and Depauw \cite{Depauw}.
In a next paper \cite{BC3} we plan to consider an intermediate case, in which space coordinates are split into different groups: along some of them the derivative is the singular integral of a measure, along the remaining of an $L^1$ function. Further implications for the Vlasov-Poisson equation, in the same spirit as in \cite{BC2}, will follow.

\subsection{Plan of the paper}

In Section~\ref{s:back} we introduce some background material about weak Lebesgue spaces, maximal functions, and singular integrals, together with an interpolation lemma. In Section~\ref{Cancellations} we introduce the smooth maximal function and we exploit the cancellation properties to prove the weak estimate for the composition of the smooth maximal function with the singular integral. Section~\ref{s:diffq} is devoted to the proof of the fact that the smooth maximal function is suited for the estimate of the difference quotients, as the usual maximal function. The core estimate is contained in Section~\ref{s:estimate}: after having introduced the notion of regular Lagrangian flow and having described its main properties, we describe the class of vector fields we are interested in and we prove the estimate for the integral functional $\Phi_\delta(s)$ (see Proposition~\ref{l:main}). Section~\ref{s:consequences} presents the corollaries of such estimate: existence, uniqueness, quantitative stability and compactness. The forward and backward flows and the Jacobian are studied. Finally, Section~\ref{s:lagrangian} contains the theory of Lagrangian solutions to the continuity and transport equations.

\subsection{About the value of the constants} In all the paper, we denote by $C$ a generic constant, whose value may vary from line to line. In particular cases, when we want to underline the dependence of the constant on relevant parameters, we use subscripts like for instance $C_N$, $C_p$ or $C_{N,p}$.

\section{Background material}\label{s:back}

This section is devoted to some classical estimates of harmonic analysis and few extensions.
Many proofs can be found in \cite{Stein}.

\subsection{Weak Lebesgue spaces}
We recall here the definition of the weak Lebesgue spaces $M^p(\Omega)$,
which are also known in the literature as Lorentz spaces, Marcinkiewicz spaces, and alternatively denoted by
$L^{p,\infty}(\Omega)$ or $L^p_w(\Omega)$.

As usual, $\Leb^N$ is the $N$-dimensional Lebesgue measure, and we shall denote by $B_R$ the ball of radius $R$ in $\R^N$ centered at the origin.

\begin{definition} Let $u$ be a measurable function defined
on an open set $\Omega \subset \R^N$. For any $1 \leq p < \infty$ we set
\begin{equation}\label{e:pn}
\lpn u \rpn_{M^p(\Omega)}^p = \sup_{\lambda >0} \Big\{ \lambda^p \Leb^N \big( \{ x \in \Omega \; : \; |u(x)| > \lambda \} \big) \Big\},
\end{equation}
and we define the {\em weak Lebesgue space} $M^p(\Omega)$ as the space consisting of all measurable functions $u : \Omega \to \R$ with $\lpn u \rpn_{M^p(\Omega)} < +\infty$. By convention, for $p=\infty$ we simply set $M^\infty(\Omega)=L^\infty(\Omega)$.
\end{definition}
In contrast to the case of Lebesgue spaces $L^p(\Omega)$, it happens that $\lpn \cdot \rpn_{M^p(\Omega)}$ is not subadditive and thus it is not a norm, hence $M^p(\Omega)$ is not a Banach space
(for this reason, we have chosen the notation $\lpn \cdot \rpn_{M^p(\Omega)}$ with three vertical bars, different
from the usual one for the norm). Nevertheless, the following inequality holds
$$ \lpn u+v \rpn_{M^p(\Omega)}^{p/(p+1)} \leq \lpn u \rpn_{M^p(\Omega)}^{p/(p+1)} + \lpn v \rpn_{M^p(\Omega)}^{p/(p+1)}, $$
and it implies in particular 
\begin{equation}\label{e:care}
\lpn u+v \rpn_{M^p(\Omega)} \leq C_p \left( \lpn u \rpn_{M^p(\Omega)} + \lpn v \rpn_{M^p(\Omega)} \right).
\end{equation}
Since for every $\lambda >0$
$$
\lambda^p \Leb^N \big( \{ x \in \Omega \; : \; |u(x)| > \lambda \} \big)
= \int_{|u|>\lambda} \lambda^p \, dx
\leq \int_{|u|>\lambda} |u(x)|^p \, dx \leq \| u \|_{L^p(\Omega)}^p,
$$
the inclusion $L^p(\Omega) \subset M^p(\Omega)$ holds, and in particular we have $\lpn u \rpn_{M^p(\Omega)} \leq \| u \|_{L^p(\Omega)}$.
This inclusion is however strict: for instance, the function $1/x$ defined on $]0,1[$ belongs to $M^1$ but not to $L^1$.

In the following lemma we show that we can interpolate $M^1$ and $M^p$,
for $p>1$, obtaining a bound on the $L^1$ norm, depending only logarithmically
on the $M^p$ norm.
\begin{lemma}
Let $u : \Omega \to [0,+\infty[$ be a nonnegative measurable function, where $\Omega \subset \R^N$ has finite measure. Then for every $1<p<\infty$ we have the interpolation estimate
\begin{equation}\label{e:interp}
\| u \|_{L^1(\Omega)} \leq  \frac{p}{p-1} \lpn u \rpn_{M^1(\Omega)} \left[ 1 +
\log \left( \frac{ \lpn u\rpn_{M^p(\Omega)}}{\lpn u \rpn_{M^1(\Omega)}} \Leb^N (\Omega)^{1-\frac{1}{p}} \right) \right]\,,
\end{equation}
and analogously for $p=\infty$
\begin{equation}\label{e:interpinfty}
\| u \|_{L^1(\Omega)} \leq  \lpn u \rpn_{M^1(\Omega)} \left[ 1 +
\log \left( \frac{ \lpn u\rpn_{L^\infty(\Omega)}}{\lpn u \rpn_{M^1(\Omega)}} \Leb^N (\Omega) \right) \right]\,.
\end{equation}
\label{Lemma InterpM1}
\end{lemma}
\begin{proof}
Setting
$$ m(\lambda) = \Leb^N \big( \{ u>\lambda \} \cap \Omega \big)\,, $$
we have the identity
\begin{equation}\label{e:lev1}
\| u \|_{L^1(\Omega)} = \int_0^{\infty} m(\lambda) \, d\lambda \,.
\end{equation}
Assume first that $p<\infty$. For every $\lambda > 0$, there holds
$$ m(\lambda) \leq \Leb^N(\Omega) \qquad \text{ and } \qquad m(\lambda) \leq \frac{ \lpn u \rpn_{M^p(\Omega)}^p}{\lambda^p} \,, $$
where the second estimate immediately follows from \eqref{e:pn}.
According to this, we now split the integral in \eqref{e:lev1}
in three parts. Let us set
$$ \alpha = \frac{\lpn u \rpn_{M^1(\Omega)}}{\Leb^N(\Omega)} \qquad \text{ and } \qquad 
\beta = \left( \frac{\lpn u \rpn_{M^p(\Omega)}^p}{\lpn u \rpn_{M^1(\Omega)}} \right)^{\frac{1}{p-1}}\,. $$
Since 
$$
\lambda \Leb^N \big( \{ u>\lambda\} \cap \Omega \big) \leq
\lambda \Leb^N \big( \{ u>\lambda\} \cap \Omega \big)^{\frac{1}{p}} \Leb^N(\Omega)^{1-\frac{1}{p}} \leq
\lpn u \rpn_{M^p(\Omega)} \Leb^N(\Omega)^{1-\frac{1}{p}}
$$
we have
$$ \lpn u \rpn_{M^1(\Omega)}\leq\lpn u\rpn_{M^p(\Omega)} \Leb^N (\Omega)^{1-\frac{1}{p}} \,, $$
so that $\alpha\leq\beta$ (we can assume that $\lpn u\rpn_{M^p(\Omega)}<\infty$).
By direct computations we have
\begin{equation}\label{e:i1}
\int_0^\alpha \Leb^N(\Omega) \, d\lambda = \lpn u \rpn_{M^1(\Omega)}\,,
\end{equation}
\begin{equation}\label{e:i2}
\int_\alpha^\beta \frac{\lpn u \rpn_{M^1(\Omega)}}{\lambda} \, d\lambda = \lpn u \rpn_{M^1(\Omega)} 
\log \left( \Leb^N (\Omega) \left( \frac{ \lpn u\rpn_{M^p(\Omega)}}{\lpn u \rpn_{M^1(\Omega)}} \right)^{\frac{p}{p-1}} \right),
\end{equation}
and 
\begin{equation}\label{e:i3}
\int_\beta^{\infty} \frac{ \lpn u \rpn_{M^p(\Omega)}^p}{\lambda^p} \, d\lambda = 
\frac{1}{p-1} \lpn u \rpn_{M^1(\Omega)}.
\end{equation}
Summing up \eqref{e:i1}, \eqref{e:i2}, \eqref{e:i3} and
recalling \eqref{e:lev1}, the desired formula \eqref{e:interp} is proved.
The case $p=\infty$ giving \eqref{e:interpinfty} is easier and is left to the reader.
\end{proof}

\subsection{Maximal function} We now introduce the concept of maximal function and present some relevant properties.

\begin{definition} Let $u$  be a measurable function defined on $\R^N$.
We define the {\em maximal function} of $u$ as
\begin{equation}
M u(x) = \sup_{\eps>0} \mint_{B_\eps(x)} |u(y)| \, dy \qquad\text{for every }x\in\R^N \,.
\label{eq:Mf}
\end{equation}
We define similarly the maximal function of a locally finite measure.
\end{definition}

\begin{propos}\label{p:classmax}
For every $1<p \leq \infty$ we have the strong estimate
\begin{equation}\label{e:maxp}
\| Mu \|_{L^p(\R^N)} \leq C_{N,p} \|u\|_{L^p(\R^N)}\,, 
\end{equation}
while for $p=1$ we have the weak estimate
\begin{equation}\label{e:max1}
\lpn Mu \rpn_{M^1(\R^N)} \leq C_N \| u \|_{L^1(\R^N)} \,. 
\end{equation}
The inequality \eqref{e:max1} also holds for a finite measure.
\end{propos}

\begin{remark} We stress the fact that the strong estimate \eqref{e:maxp} does not hold when $p=1$. It is even possible to show that, if $u \in L^1(\R^N)$ and $u \not \equiv 0$, then $Mu \not \in L^1(\R^N)$.
Nevertheless, given $u\in L^1(\R^N)$, we have $Mu\in L^1_{\loc}(\R^N)$
if and only if $|u|\log^+|u|\in L^1_{\loc}(\R^N)$,
where $\log^+ t = \max \{\log t, 0\}$.
\end{remark}

\begin{lemma}\label{l:lemmapsi}
Let $\psi : ]0,\infty[ \to [0,\infty[$ be a nonincreasing function and assume that
$$
I \equiv \int_{\R^N} \psi (|y|) \, dy < \infty \,. 
$$
Then for every $u \in L^1_\loc (\R^N)$ and every $\eps > 0$ we have
$$ \int_{\R^N} | u(x-y) | \frac{1}{\eps^N} \psi \left( \frac{ |y| }{\eps} \right) \, dy \leq I \cdot Mu (x) 
\qquad \text{\rm for every $x \in \R^N$.} $$
\end{lemma}
\noindent The proof of this lemma can be found in \cite{Stein},
Chapter III, Section 2.2, Theorem 2(a).

\subsection{Singular integral operators}\label{ss:singop} We now present different classes of singular kernels and describe the properties of the associated singular integral operators. As usual, $\S'(\R^N)$ is the space of tempered distributions on $\R^N$, and $\S(\R^N)$ the Schwartz space.

\begin{definition}[Singular kernel]\label{d:ker}
We say that $K$ is a {\em singular kernel} on $\R^N$ if
\begin{itemize}
\item[(i)] $K \in \S'(\R^N)$ and $\widehat K \in L^\infty (\R^N)$;
\item[(ii)] $K | _{\R^N \setminus \{0\}} \in L^1_\loc (\R^N \setminus \{0\})$ and there exists a constant $A \geq 0$ such that
$$
\int_{|x| > 2|y|} |K(x-y) - K(x)| \, dx \leq A \qquad \text{for every $y \in \R^N$.} $$
\end{itemize}
\end{definition}

\begin{theorem}[Calder\'on--Zygmund]\label{t:CZint} Let $K$ be a singular kernel and define
$$ Su = K * u \qquad \text{ for $u \in L^2(\R^N)$} \,, $$
in the sense of multiplication in the Fourier variable
(recall (i) in Definition~\ref{d:ker}).
Then for every $1<p < \infty$ we have the strong estimate
\begin{equation}\label{e:defintp}
\| Su \|_{L^p(\R^N)} \leq C_{N,p}(A+\|\widehat K\|_{L^\infty}) \|u\|_{L^p(\R^N)}\,, \qquad u \in L^p\cap L^2(\R^N)\,,
\end{equation}
while for $p=1$ we have the weak estimate
\begin{equation}\label{e:defint1}
\lpn Su \rpn_{M^1(\R^N)} \leq C_{N}(A+\|\widehat K\|_{L^\infty}) \| u \|_{L^1(\R^N)} \,, \qquad u \in L^1\cap L^2(\R^N) \,.
\end{equation}
\end{theorem}

\begin{corol}\label{c:CZint}
The operator $S$ can be extended to the whole $L^p(\R^N)$ for any $1<p< \infty$, with values in $L^p(\R^N)$, and estimate \eqref{e:defintp} holds for every $u \in L^p(\R^N)$. Moreover, the operator $S$ can be extended to the whole $L^1(\R^N)$, with values in $M^1(\R^N)$, and estimate \eqref{e:defint1} holds for every $u \in L^1(\R^N)$.
\end{corol}

\begin{definition}
The operator $S$ constructed in Corollary~\ref{c:CZint} is called the {\em singular integral operator} associated to the singular kernel $K$.
\end{definition}

\begin{remark} The case $p=1$ deserves some comments.
Indeed, the extension $S^{M^1}$ defined on $L^1$ with values in $M^1$
can induce some confusion, due to the fact
that a function in $M^1$ is in general not locally integrable, thus
does not define a distribution. We can observe that for $u\in L^1(\R^N)$,
we can define a tempered distribution $S^D u\in\S'(\R^N)$ by the formula
\begin{equation}
	\left\langle S^D u,\varphi\right\rangle=\left\langle u,\widetilde S\varphi\right\rangle
	\qquad\text{for every }\varphi\in\S(\R^N),
	\label{eq:S^D}
\end{equation}
where $\widetilde S$ is the singular integral operator associated to the kernel
$\widetilde K(x)=K(-x)$. Indeed, for $\varphi\in\S(\R^N)$, we have
$\widetilde S\varphi\in H^q(\R^N)\subset C_0(\R^N)$ for $q>N/2$.
Then $S^D:L^1(\R^N)\rightarrow \S'(\R^N)$ is an extension of $S$,
with values tempered distributions. The operators $S^{M^1}$ and $S^D$
are different and cannot be identified.
Observe also that $S^D u\in\S'(\R^N)$ can also be defined by
\eqref{eq:S^D} for $u$ a finite measure on $\R^N$. Also notice that
the definition in \eqref{eq:S^D} is equivalent to the definition in
Fourier variables
$$
\widehat{S^Du} = \widehat K \widehat u \,,
$$
for which we use that $\widehat K \in L^\infty$ and $\widehat u \in L^\infty$.
\label{Rk extM1D}
\end{remark}

The following characterization of singular kernels is available.

\begin{propos}Consider a function $K\in L^1_{\loc}(\R^N \setminus \{0\})$ satisfying the 
following conditions:
\begin{itemize}
\item[(i)] There exists a constant $A \geq 0$ such that
$$ \int_{|x| > 2|y|} |K(x-y) - K(x)| \, dx \leq A \qquad \text{for every $y \in \R^N$}; $$
\item[(ii)] There exists a constant $A_0\ge 0$ such that
$$ \int_{|x|\leq R}|x||K(x)|\,dx\leq A_0R \qquad\text{for every }R>0; $$
\item[(iii)] There exists a constant $A_2 \geq 0$ such that
$$
\left| \int_{R_1 < |x| < R_2} K(x) \, dx \right| \leq A_2
\qquad\text{for every }0 < R_1 < R_2 < \infty.
$$
\end{itemize}
Then $K$ can be extended to a tempered distribution on $\R^N$ which
is a singular kernel, unique up to a constant times a Dirac mass
at the origin. Conversely, any singular kernel on $\R^N$ has
a restriction on $\R^N \setminus \{0\}$ that satisfies the previous
conditions (i), (ii), (iii).
\label{p:singkernels}
\end{propos}
For our purpose we introduce a more regular class of kernels.

\begin{definition}\label{d:fundker} 
We say that a kernal $K$ is a {\em singular kernel of fundamental type} if the following properties hold:
\begin{itemize}
\item[(i)] $K |_{\R^N \setminus \{0\}} \in C^1(\R^N \setminus \{0\})$;
\item[(ii)] There exists a constant $C_0 \geq 0$ such that
\begin{equation}\label{e:fund0}
|K(x)| \leq \frac{C_0}{|x|^N} \qquad \text{ for every $x \not = 0$;}
\end{equation}
\item[(iii)] There exists a constant $C_1 \geq 0$ such that
\begin{equation}\label{e:fund1}
|\nabla K(x)| \leq \frac{C_1}{|x|^{N+1}} \qquad \text{ for every $x \not = 0$;}
\end{equation}
\item[(iv)] There exists a constant $A_2 \geq 0$ such that
\begin{equation}\label{e:fund2}
\left| \int_{R_1 < |x| < R_2} K(x) \, dx \right| \leq A_2
\qquad\text{for every }0 < R_1 < R_2 < \infty.
\end{equation}
\end{itemize}
These conditions imply those in Proposition~\ref{p:singkernels}.
\end{definition}

\subsection{An interpolation lemma}

The following interpolation lemma is a generalization
of classical results on singular integrals, see for instance Section II.2 of \cite{Stein}.
We give its full proof for completeness.

\begin{lemma}\label{l:interp}
Let $T_+ : L^2(\R^N) \to L^2(\R^N)$ be a (nonlinear) operator satisfying
\begin{itemize}
\item[(i)] $T_+ (u) \geq 0$ for every $u \in L^2(\R^N)$;
\item[(ii)] $T_+(u+v) \leq T_+(u) + T_+(v)$ for every $u, v \in L^2(\R^N)$;
\item[(iii)] $T_+(\lambda u) = |\lambda| T_+(u)$ for every $u \in L^2(\R^N)$ and every $\lambda \in \R$;
\item[(iv)] There exists a constant $P_2\ge 0$ such that
$$ \| T_+(u)\|_{L^2(\R^N)} \leq P_2 \|u\|_{L^2(\R^N)}
\qquad\text{for every $u \in L^2(\R^N)$;} $$
\item[(v)] There exists a constant $P_1\ge 0$ such that
if $u \in L^2(\R^N)$ satisfies $\spt u \subset \overline{B}_R(x_0)$
for some $x_0\in\R^N$ and $R>0$, and $\int_{\R^N} u = 0$, then
$$
\int_{|x-x_0| > 2R} T_+(u) \, dx \leq P_1 \|u\|_{L^1(\R^N)} \,.
$$
\end{itemize}
Then there exists a constant $C_N$, which depends only on the dimension $N$, such that
$$
	\lpn T_+ (u) \rpn_{M^1(\R^N)} \leq C_N(P_1+P_2) \|u\|_{L^1(\R^N)}
	\qquad\text{for every }u \in L^1 \cap L^2 (\R^N).
$$
\end{lemma}

\begin{proof}
We preliminarily notice that from assumptions (i), (ii), (iii) it follows that
\begin{equation}\label{e:prelsub}
T_+(-u) = T_+(u), \qquad |T_+(u) - T_+(v) | \leq T_+(u-v),
\end{equation}
for every $u, v \in L^2(\R^N)$.

\medskip

{\sc Step 1. Calder\'on--Zygmund decomposition.} Given $u \in L^1(\R^N)\cap L^2 (\R^N)$
and $\alpha>0$, we perform the so-called Calder\'on--Zygmund decomposition in cubes of $\R^N$ (see for instance Section I.3 of \cite{Stein}).
We find a family $\{I_k\}_{k=1}^\infty$ of closed cubes with disjoint interiors such that
$$
\alpha \Leb^N (I_k) < \int_{I_k} |u| \leq 2^N \alpha \Leb^N(I_k) \qquad \text{ for every $k$},
$$
and
$$
|u| \leq \alpha \qquad \text{a.e. outside } \cup_k I_k \,.
$$
We then set
\begin{equation}\label{e:defwk}
w_k = \left( u - \mint_{I_k} u \right) \ind_{I_k}
\end{equation}
and
\begin{equation}\label{e:defv}
v = \left\{
\begin{array}{cl} u & \text{ for $x \not \in \cup_k I_k$}, \\
\displaystyle\mint_{I_k} u & \text{ for $x \in I_k$.} \end{array} \right.
\end{equation}
Then we obviously have $v,w_k\in L^1(\R^N)$,
\begin{equation}\label{e:propw}
\spt w_k \subset I_k \,, \quad \int_{I_k} w_k = 0 \,, \quad \|w_k\|_{L^1(\R^N)} \leq 2 \int_{I_k} |u| \,, \quad\sum_k \|w_k\|_{L^1(\R^N)} \leq 2 \| u \|_{L^1(\R^N)}.
\end{equation}
Moreover, it is readily checked that
\begin{equation}
	u = v + \sum_k w_k \qquad \text{$\Leb^N$-a.e. and in $L^1(\R^N)$,}
	\label{eq:ueqsum}
\end{equation}
and that
\begin{equation}\label{e:propv}
\| v \|_{L^1(\R^N)} \leq \| u \|_{L^1(\R^N)} \,, \quad\quad  \| v \|_{L^\infty(\R^N)} \leq 2^N \alpha\,.
\end{equation}
We notice also that
\begin{equation}
	\Leb^N\Bigl(\cup_k I_k\Bigr)\leq\sum_k\Leb^N(I_k)\leq \frac{1}{\alpha}\| u \|_{L^1(\R^N)}.
	\label{eq:measUIk}
\end{equation}
Now, for every $k$ consider an open ball $B_k\equiv B_{r_k} (y_k)$ containing $I_k$ and with the same center $y_k$, such that for some dimensional constant $\beta_N$ we have
$$ \Leb^N(B_k) \leq \beta_N \Leb^N(I_k) \,. $$
Moreover, we set
$$ V_k = B_{2r_k} (y_k), \qquad V = \cup_k V_k \,. $$
Then, similarly as in \eqref{eq:measUIk}, we have
\begin{equation}\label{e:measV}
\Leb^N (V) \leq \sum_k \Leb^N (V_k) \leq \sum_k 2^N \beta_N\Leb^N (I_k) \leq 2^N \beta_N \frac{1}{\alpha} \| u \|_{L^1(\R^N)} \,.
\end{equation}
Since by \eqref{e:propw} $\spt w_k \subset \overline{B}_{k}$ and $\int_{\R^N} w_k = 0$,
from assumption (v) we get
\begin{equation}\label{e:propT}
\int_{\R^N \setminus \overline {V_k}} T_+ (w_k) \leq P_1 \|w_k\|_{L^1(\R^N)} \,.
\end{equation}

\medskip

{\sc Step 2. Proof of the weak estimate.} We fix an arbitrary $m \in \N$ and we first estimate 
$T_+ \left( v + \sum_{k=1}^m w_k \right)$. From (ii) it follows that
\begin{equation}\label{e:su1}
T_+ \left( v + \sum_{k=1}^m w_k \right) \leq T_+ (v) + \sum_{k=1}^m T_+(w_k) \,.
\end{equation}
{}From \eqref{e:propT} and \eqref{e:propw} we deduce that
\begin{equation}\label{e:su2}
\left\| \sum_{k=1}^m T_+ (w_k) \right\|_{L^1(\R^N \setminus V)} \leq \sum_{k=1}^m P_1 \| w_k \|_{L^1(\R^N)} \leq 2P_1 \| u \|_{L^1(\R^N)} \,.
\end{equation}
Moreover, noticing that \eqref{e:propv} implies
$$
\| v \|_{L^2(\R^N)} \leq \left( 2^N \alpha \|u\|_{L^1(\R^N)} \right)^{1/2} \,,
$$
and using assumption (iv), we obtain
\begin{equation}\label{e:su3}
\left\| T_+ (v) \right\|_{L^2(\R^N)} \leq P_2 \left( 2^N \alpha \|u\|_{L^1(\R^N)} \right)^{1/2} \,.
\end{equation}
For every $\lambda > 0$ we can estimate
\begin{eqnarray}\label{e:estsuplev}
&& \Leb^N \left( \left\{ x \in \R^N \; : \; T_+ \biggl( v + \sum_{k=1}^m w_k \biggr)(x) > \lambda \right\} \right) \nonumber \\
& \stackrel{\eqref{e:su1}}{\leq} & \Leb^N \left( \left\{ x \in \R^N \; : \; T_+(v)(x) > \frac{\lambda}{2} \right\} \right) +
\Leb^N \left( \left\{ x \in V \; : \; \sum_{k=1}^m T_+ (w_k)(x) > \frac{\lambda}{2} \right\} \right) \nonumber  \\
&& + \Leb^N \left( \left\{ x \not \in V \; : \; \sum_{k=1}^m T_+ (w_k)(x) > \frac{\lambda}{2} \right\} \right) \\
& \stackrel{\eqref{e:su3}, \eqref{e:su2}}{\leq} & \frac{1}{(\lambda/2)^2} P_2^2 2^N \alpha \| u \|_{L^1(\R^N)} 
+ \Leb^N(V) + \frac{1}{\lambda/2} 2 P_1 \| u \|_{L^1(\R^N)} \nonumber \\
& \stackrel{\eqref{e:measV}}{\leq} & \left[ \frac{\alpha}{\lambda^2} 2^{N+2} P_2^2 
+ \frac{1}{\alpha} 2^N \beta_N  + \frac{1}{\lambda} 4 P_1 \right] \| u \|_{L^1(\R^N)} \,.\nonumber 
\end{eqnarray}
Now, according to \eqref{e:defwk}, \eqref{e:defv} we have 
$$
\left| \sum_{k=1}^m w_k \right| \leq |u| + |v| \in L^2 (\R^N),
$$
and since by \eqref{eq:ueqsum} $\sum_{k=1}^m w_k \to \sum_{k=1}^\infty w_k=u-v$ a.e.~in $\R^N$,
we deduce by Lebesgue's theorem that, as $m\rightarrow \infty$,
$$
v + \sum_{k=1}^m w_k \to v + \sum_{k=1}^\infty w_k = u \qquad \text{ in $L^2(\R^N)$.}
$$ 
According to assumption (iv) and to \eqref{e:prelsub}, this implies that
$$
T_+ \left( v + \sum_{k=1}^m w_k \right) \to T_+ (u) \qquad \text{ in $L^2(\R^N)$.}
$$
Then (up to the extraction of a subsequence) we also have $T_+ \left( v + \sum_{k=1}^m w_k \right) 
\to T_+ (u)$ pointwise a.e.~in $\R^N$, and this implies that
$$ \ind_{\{ x \, : \, T_+ ( u )(x) > \lambda \}}
\leq\liminf_{m\rightarrow\infty}\,\ind_{\{ x \, : \, T_+ ( v + \sum_{k=1}^m w_k )(x) > \lambda \}}
\qquad\text{for a.e. } x\in\R^N. $$
Using Fatou's lemma and \eqref{e:estsuplev}, we get
\begin{equation}\label{e:toopt}
\Leb^N \left( \left\{ x \in \R^N \; : \; T_+ \left( u \right)(x) > \lambda \right\} \right) 
\leq \left[ \frac{\alpha}{\lambda^2} 2^{N+2} P_2^2 
+ \frac{1}{\alpha} 2^N \beta_N  + \frac{1}{\lambda} 4 P_1 \right] \| u \|_{L^1(\R^N)} \,.
\end{equation}
Since $\lambda,\alpha > 0$ are arbitrary, we choose $\alpha$ in order to optimize the estimate \eqref{e:toopt},
by setting
$ \alpha = \sqrt{\beta_N} \lambda/2 P_2$.
This yields that for any $\lambda>0$
$$ \Leb^N \left( \left\{ x \in \R^N \; : \; T_+ \left( u \right)(x) > \lambda \right\} \right) 
\leq \frac{1}{\lambda} \left[ 2^{N+2} \sqrt{\beta_N} P_2 + 4P_1 \right] \| u \|_{L^1(\R^N)} \,, $$
which is the thesis of the lemma.
\end{proof}

\section{Cancellations in maximal functions and singular integrals}\label{Cancellations}

In this section we provide a key estimate that states that
there are some cancellations in the composition of a singular
integral operator and a maximal function.

The idea of such cancellation is the following.
To simplify, consider singular integral operators associated to smooth kernels $K$,
such that $\widehat K\in C^\infty(\R^N\backslash\{0\})$, with
\begin{equation}
	\Bigl|\partial_\alpha\widehat K(\xi)\Bigr|\leq\frac{C_{|\alpha|}}{|\xi|^{|\alpha|}}
	\qquad\text{for all }\alpha\in\N^N \text{ and }\xi\in\R^N\backslash\{0\}.
	\label{eq:Fouriermultiplier}
\end{equation}
It is well-known that such $K$ is a singular kernel satisfying the conditions of Definition~\ref{d:fundker},
thus to $K$ we can associate a singular integral operator $S$, that
satisfies the weak estimate \eqref{e:defint1}. Now, if $K_1$ and
$K_2$ are two such operators, we can consider the composition
$S_2S_1$. Then, we can see that for all $u\in L^2$,
$S_2S_1u=Su$, where $S$ is associated to the kernel $K$ defined
by $\widehat K=\widehat K_2\widehat K_1$, that
again satisfies \eqref{eq:Fouriermultiplier}. Therefore,
$S_2S_1$ also satisfies the weak estimate \eqref{e:defint1}.
However, it is not possible to get this information just by
composition, since when $u\in L^1\cap L^2$, $S_1u$ is not controlled
in $L^1$ by the $L^1$ norm of $u$. The explanation of this
phenomenon lies in the cancellations that hold in the composition
$S_2S_1$ (i.e. in the formal convolution $K_2*K_1$), due to 
condition \eqref{e:fund2}.

The main result of this section, Theorem~\ref{t:mainest} states
that the same kind of cancellation occurs in the composition of
a singular integral operator by a maximal function.
However, the usual maximal function \eqref{eq:Mf} is too rough to allow
such cancellation. Therefore we consider now {\em smooth maximal functions},
and moreover we put the absolute value {\em outside the integral},
instead of inside the integral as in \eqref{eq:Mf}. This smooth maximal function
is also known as {\em grand maximal function}, and is an important
tool in the theory of Hardy spaces (see for instance \cite{semmes}).

\begin{definition}\label{d:rhomax} Given a family of functions $\{\rho^\nu\}_\nu \subset L^\infty_c(\R^N)$, for every function $u \in L^1_\loc (\R^N)$ we define the {\em $\{\rho^\nu\}$-maximal function} of $u$ as
\begin{equation}
	M_{\{\rho^\nu\}} (u) (x) = \sup_\nu \sup_{\eps > 0} \left| \int_{\R^N} \rho^\nu_\eps (x-y) u(y) \, dy \right|
	=\sup_\nu \sup_{\eps > 0} \Bigl|(\rho^\nu_\eps*u)(x)\Bigr|
	\qquad \text{ for every $x \in \R^N$},
	\label{eq:maxrho}
\end{equation}
where we use the notation
$$ \rho^\nu_\eps (x)\equiv\frac{1}{\eps^N}\rho^\nu\left(\frac{x}{\eps}\right) \,. $$
In the case when $\{\rho^\nu\}_\nu \subset C^\infty_c(\R^N)$,
we can use the same definition in the case of distributions $u \in \D'(\R^N)$,
more precisely we set
$$ M_{\{\rho^\nu\}} (u) (x) = \sup_\nu \sup_{\eps > 0} \Bigl| \langle u , \rho^\nu_\eps (x-\cdot) \rangle  \Bigr|
	\qquad \text{ for every $x \in \R^N$.} $$
\end{definition}
\begin{remark} Taking $\rho^\nu(x)=\ind_{|x|<1}/\Leb^N(B_1)$ in \eqref{eq:maxrho} gives
the maximal function \eqref {eq:Mf}, except that now the absolute
value is outside the integral.
\label{Rk weightindicator}
\end{remark}

The announced cancellation between the singular integral and this maximal function are described in the following theorem.

\begin{theorem}\label{t:mainest}
Let $K$ be a singular kernel of fundamental type as in Definition~\ref{d:fundker} and set $Su = K * u$ for every $u \in L^2(\R^N)$.
Let $\{\rho^\nu\}_\nu \subset L^\infty(\R^N)$ be a family of kernels such that 
\begin{equation}\label{e:unifrho}
\spt \rho^\nu \subset B_1 \qquad\text{and}\qquad \|\rho^\nu\|_{L^1(\R^N)} \leq Q_1 \qquad 
\text{for every $\nu$.}
\end{equation}
Assume that for every $\eps > 0$ and for every $\nu$, there holds
$\Big( \eps^N K(\eps \cdot) \Big) * \rho^\nu \in C_b(\R^N)$
with the norm estimate
\begin{equation}\label{e:normcb}
\left\| \Big( \eps^N K(\eps \cdot) \Big) * \rho^\nu \right\|_{C_b(\R^N)} \leq Q_2
\qquad\text{for every }\eps > 0 \text{ and every }\nu.
\end{equation}
Then the following estimates hold:
\begin{itemize}
\item[(i)] There exists a constants $C_N$, depending on the dimension $N$ only, such that
\begin{equation}\label{e:weakest}
\lpn M_{\{\rho^\nu\}} (Su) \rpn_{M^1(\R^N)} 
\leq C_N\Bigl( Q_2 + Q_1\bigl( C_0 + C_1+\|\widehat K\|_\infty\bigr) \Bigr) \|u\|_{L^1(\R^N)}
\quad \text{for every $u \in L^1 \cap L^2 (\R^N)$;}
\end{equation}
\item[(ii)] If in addition $\{\rho^\nu\} \subset C^\infty_c (\R^N)$,
the estimate \eqref{e:weakest} holds for all $u$ finite measure
on $\R^N$, with the same constant $C_N$,
where $Su$ is defined as a distribution, according to \eqref{eq:S^D};
\item[(iii)] If $Q_3\equiv\sup_\nu \| \rho^\nu \|_{L^\infty(\R^N)}$ is finite,
then there exists a constant $C_N$, depending on the dimension $N$, such that
\begin{equation}\label{e:strongest}
\left\| M_{\{\rho^\nu\}} (Su) \right\|_{L^2(\R^N)} 
\leq C_NQ_3 \|\widehat K\|_\infty \|u\|_{L^2(\R^N)}
\qquad \text{for every $u \in L^2 (\R^N)$.}
\end{equation}
\end{itemize}
\end{theorem}

\begin{remark}
The assumption \eqref{e:normcb} on $\rho^\nu$ is a regularity assumption.
For instance, it is satisfied if $\rho^\nu \in H^q$ for some $q > N/2$ with
uniform bounds.
Indeed, in this case we have $\widehat{\rho}^\nu\in L^1$
with uniform bounds, thus since $\widehat K\in L^\infty$, we get that
$\widehat K(\xi/\eps)\widehat{\rho}^\nu(\xi)\in L^1$
with uniform bounds.
\label{Rk Hq}
\end{remark}

\begin{proof}[Proof of Theorem~\ref{t:mainest}]

\medskip

{\sc Step 1. Definition of the quantity $\Delta_\eps^\nu $ and estimates.}
Fix a radial function $\chi \in C^\infty(\R^N)$, $0\leq\chi\leq 1$, such that $\chi(x) = 0$ for $|x| \leq 1/2$, $\chi(x) = 1$ for $|x| \geq 1$, and with $\| \nabla \chi\|_\infty \leq 3$. We define
\begin{equation}\label{e:deltaeps}
\Delta_\eps^\nu  (x) = \Big[ K(\eps\cdot) * \rho^\nu  \Big] \left( \frac{x}{\eps} \right) - \left( \int_{\R^N} \rho^\nu (y) \, dy \right) \chi\left( \frac{x}{\eps} \right) K(x) \,.
\end{equation}
{}From assumption \eqref{e:normcb}, the definition of $\chi$ and assumptions \eqref{e:unifrho} and \eqref{e:fund0}, we deduce that for every $\eps>0$ and every $\nu$,
we have $\Delta_\eps^\nu\in C_b(\R^N)$ and
\begin{equation}\label{e:deltainfty}
| \Delta_\eps^\nu  (x) | \leq \frac{Q_2}{\eps^N} + \frac{Q_1 C_0 2^N}{\eps^N} \qquad \text{for every $x \in \R^N$.}
\end{equation}
For $|x| > 2\eps$ we have
$$ \Delta_\eps^\nu  (x) = \int_{\R^N} \big[ K(y) - K(x) \big] \rho^\nu _\eps (x-y) \, dy \,, $$
and this implies the estimate
\begin{equation}\label{e:deltaone}
|\Delta_\eps^\nu  (x) | \leq \int_{\R^N} \frac{C_1 |y-x|}{(|x| / 2)^{N+1}} |\rho^\nu _\eps (x-y)| \, dy \leq \frac{2^{N+1} C_1 Q_1\eps}{|x|^{N+1}},
\qquad \text{for every $x \in \R^N$ with $|x| > 2\eps$}.
\end{equation}
In obtaining \eqref{e:deltaone} we have used assumptions \eqref{e:fund1} and \eqref{e:unifrho},
the fact that the integral is indeed performed on the set $B_\eps(x)$,
and that for all $x$ and $y$ under consideration we have
for every $s \in [0,1]$
$$ |x+s(y-x)| \geq|x|-|y-x|\geq|x|-\eps\geq|x|-|x|/2= |x| / 2 \,. $$
Putting together \eqref{e:deltainfty} and \eqref{e:deltaone},
we obtain the existence of a dimensional constant $C_N$ such that
$$
| \Delta_\eps^\nu  (x) | \leq 
C_N \frac{Q_2 + Q_1(C_0 + C_1)}{\eps^N \left( 1 + \left( \frac{|x|}{\eps} \right)^{N+1} \right)} \,
\qquad \text{for every $x \in \R^N$}.
$$
This in particular gives $\Delta_\eps^\nu  \in L^1 \cap L^\infty (\R^N)$.
Applying Lemma~\ref{l:lemmapsi} with
$$ \psi(z) = C_N \frac{Q_2 + Q_1(C_0 + C_1)}{ \left( 1 + z^{N+1} \right)} \,, $$
we deduce the existence of a dimensional constant $C_N$ such that
for every $u \in L^1_\loc(\R^N)$,
\begin{equation}\label{e:deltaest}
\sup_{\nu}\sup_{\eps > 0} |(\Delta^\nu _\eps * u)(x)| \leq C_N \big( Q_2 + Q_1( C_0 + C_1 ) \big) Mu (x)
\qquad \text{for every $x \in \R^N$}.
\end{equation}
Thus, recalling Proposition~\ref{p:classmax}, we deduce that
\begin{equation}\label{e:weakdelta}
\lpn \sup_{\nu} \sup_{\eps > 0} |\Delta^\nu _\eps * u| \rpn_{M^1(\R^N)}
\leq C_N \big( Q_2 + Q_1( C_0 + C_1 ) \big) \| u \|_{L^1(\R^N)} 
\qquad \text{for every $u \in L^1(\R^N)$}
\end{equation}
and
$$
\left\| \sup_{\nu} \sup_{\eps > 0} |\Delta^\nu _\eps * u| \right\|_{L^2(\R^N)}
\leq C_N \big( Q_2 + Q_1( C_0 + C_1 ) \big) \| u \|_{L^2(\R^N)} 
\qquad \text{for every $u \in L^2(\R^N)$}.
$$

\medskip

{\sc Step 2. Definition of the operator $T_+$, interpolation lemma and conclusion of the proof.}
We first notice that from the definition of $\Delta_\eps^\nu $ in \eqref{e:deltaeps}
it follows that
\begin{equation}\label{e:split}
\big( \rho_\eps^\nu  * K \big) (x) = \Delta_\eps^\nu  (x) + \left( \int_{\R^N} \rho^\nu (y) \, dy \right)
\chi \left( \frac{x}{\eps} \right) K(x)\qquad\text{for every }x\in\R^N.
\end{equation}
Then, since $\rho_\eps^\nu  * K\in L^2$, we have
\begin{equation}\label{e:associativity}
\rho_\eps^\nu * \left( Su \right) = \left( \rho_\eps^\nu * K \right) * u
\qquad\text{for every $u \in L^2(\R^N)$},
\end{equation}
as can be easily seen by using the Fourier transform.
Therefore, according to \eqref{eq:maxrho},
\begin{equation}
	M_{\{\rho^\nu\}} (Su)
	=\sup_\nu \sup_{\eps > 0} \Bigl|\rho^\nu_\eps*(Su)\Bigr|
	=\sup_\nu \sup_{\eps > 0} \Bigl|\left( \rho_\eps^\nu * K \right) * u\Bigr|
	\qquad\text{for every $u \in L^2(\R^N)$},
	\label{eq:MSu}
\end{equation}
and the left-hand side of \eqref{e:split} is precisely the kernel we have to study.
The term $\Delta_\eps^\nu$ in \eqref{e:split} has been treated in Step 1.
Thus since $\left| \int \rho^\nu (y) \, dy \right| \leq Q_1$,
it remains to study the operator 
\begin{equation}\label{e:opt}
T_+ (u) \equiv \sup_{\eps > 0} \biggl| \left( \chi\left( \frac{\cdot}{\eps} \right) K \right) * u \biggr|
\qquad \text{ for } u\in L^2(\R^N).
\end{equation}
We are going to apply the interpolation Lemma~\ref{l:interp} to the operator $T_+$.
The only assumptions of the lemma that are not immediate are (iv) and (v),
which will be checked in Steps 4 and 3 respectively, obtaining constants
$P_1=C_N(C_0+C_1)$ and $P_2=C_N(C_0+C_1+\|\widehat K\|_\infty)$.
Assuming for a moment these estimates, Lemma~\ref{l:interp} yields that
\begin{equation}\label{e:weakT}
\lpn T_+ (u) \rpn_{M^1(\R^N)} \leq C_N\Bigl(C_0+C_1+\|\widehat K\|_\infty\Bigr) \| u \|_{L^1(\R^N)}
\qquad\text{ for every }u\in L^1\cap L^2(\R^N).
\end{equation}
This, together with \eqref{eq:MSu}, \eqref{e:split}, \eqref{e:weakdelta}
and recalling \eqref{e:care}, gives \eqref{e:weakest} and proves (i).

Then, in order to prove (ii), given a finite measure $u$,
take a smoothing sequence $\zeta_n$ and define $u_n\equiv\zeta_n * u$.
Then since $u_n\in L^1\cap L^2(\R^N)$,
we can apply \eqref{e:weakest} to $u_n$.
We observe that $Su_n\rightarrow Su$ in $\S'(\R^N)$. Therefore, for fixed $\eps$, $\nu$
and $x$, $(\rho_\eps^\nu * (Su_n))(x)\rightarrow (\rho_\eps^\nu * (Su))(x)$
as $n\rightarrow\infty$. This implies that
$$ \ind\!\left\{\sup_\nu\sup_{\eps>0}\Bigl|\rho_\eps^\nu * (Su)\Bigr|>\lambda\right\}
	\leq\liminf_{n\rightarrow\infty}\,\ind\!\left\{\sup_\nu\sup_{\eps>0}\Bigl|\rho_\eps^\nu * (Su_n)\Bigr|>\lambda\right\}
	\qquad\text{for all }\lambda>0 \,.
$$
By applying Fatou's lemma, we conclude that property (ii) holds.

Finally, to show (iii), we observe that
$$ | \rho^\nu (x) | \leq Q_3\, \ind_{B_1}(x) \qquad \text{ for a.e.~$x \in \R^N$,} $$
which implies according to Remark~\ref{Rk weightindicator} that for all $u\in L^2(\R^N)$
\begin{equation}\label{e:forstrongest}
M_{\{\rho^\nu\}} (u)(x) \leq Q_3 \Leb^N(B_1) Mu (x)
\qquad\text{ for every $x\in\R^N$}.
\end{equation}
It follows from Proposition~\ref{p:classmax} that
$$ \|M_{\{\rho^\nu\}} (u)\|_{L^2}\leq Q_3 \Leb^N(B_1) \|Mu\|_{L^2}
	\leq C_NQ_3\|u\|_{L^2}
	\qquad\text{for every }u\in L^2(\R^N) \,. $$
Combining this with the trivial estimate
$\|Su\|_2\leq\|\widehat K\|_\infty\|u\|_2$ yields \eqref{e:strongest}
and property (iii).

\medskip

{\sc Step 3. Checking of assumption (v) of Lemma~\ref{l:interp}.}
In this step we show the existence of a constant $P_1$ such that
if $u \in L^2(\R^N)$ satisfies $\spt u \subset \overline{B}_R(x_0)$
and $\int_{\R^N} u = 0$ then
\begin{equation}
	\int_{|x-x_0| > 2R} T_+(u) \, dx \leq P_1 \|u\|_{L^1(\R^N)} \,.
	\label{eq:estT+1}
\end{equation}
Since $\spt u \subset \overline{B}_R(x_0)$ and $\int_{\R^N} u=0$,
we can write for $x$ such that $|x-x_0|>2R$,
\begin{equation}\begin{array}{l}
	\dsp\hphantom{=\ } \Biggl(\left( \chi\Bigl( \frac{\cdot}{\eps} \Bigr) K \right) * u\Biggr) (x) \\
\dsp = \int_{|y-x_0| \leq R} \left[ \chi\left( \frac{x-y}{\eps} \right) K(x-y) - \chi \left( \frac{x-x_0}{\eps} \right) K(x-x_0) \right] u(y) \, dy \\
\dsp= \int_{|y-x_0| \leq R} \chi\left( \frac{x-y}{\eps} \right) \Bigl[  K(x-y) - K(x-x_0)\Bigr] u(y) \, dy \\
\dsp\hphantom{=\ } + \int_{|y-x_0| \leq R} \left[ \chi\left( \frac{x-y}{\eps} \right) - \chi\left( \frac{x-x_0}{\eps} \right) \right]  K(x-x_0)  u(y) \, dy \,.
	\label{eq:estregchi}
	\end{array}
\end{equation}
In order to get an estimate for \eqref{e:opt},
we are going to estimate separately the two terms
in the last line in \eqref{eq:estregchi}.
We are interested only in those $x$ and $y$
which satisfy $|x-x_0|>2R$ and $|y-x_0| \leq R$,
which implies that for any $0\leq s\leq 1$, we have
\begin{equation}
	|x-x_0 +s(x_0-y)|\geq |x-x_0|-|x_0-y|\geq |x-x_0|-R
	\geq |x-x_0|-|x-x_0|/2=|x-x_0|/2.
	\label{eq:estx-y}
\end{equation}
\noindent For the first term in \eqref{eq:estregchi},
using \eqref{e:fund1} we can estimate the variation of $K$ as
$$ | K(x-y) - K(x-x_0) | \leq \int_0^1 \frac{ C_1 |y-x_0| }{ |x-x_0 +s(x_0-y)|^{N+1} } \, ds \leq \frac{2^{N+1} C_1 R }{|x-x_0|^{N+1}} \,. $$
Thus we obtain
\begin{equation}\begin{array}{l}
	\dsp\hphantom{\leq\ }  \left| \int_{|y-x_0| \leq R} \chi\left( \frac{x-y}{\eps} \right) \Bigl[  K(x-y) - K(x-x_0)\Bigr] u(y) \, dy \right| \\
\dsp \leq  \int_{|y-x_0| \leq R} \frac{2^{N+1} C_1 R }{|x-x_0|^{N+1}} |u(y)| \, dy 
\;\; = \;\;  \frac{2^{N+1} C_1 R }{|x-x_0|^{N+1}}\| u \|_{L^1(\R^N)} \,.
	\label{eq:estKvari}
	\end{array}
\end{equation}
In order to estimate the second term, we first notice that
\begin{equation}
	\left| \chi\left( \frac{x-y}{\eps} \right) - \chi \left( \frac{x-x_0}{\eps} \right) \right| \leq \| \nabla \chi\|_\infty \frac{R}{\eps} \,.
	\label{eq:esttranschi}
\end{equation}
Moreover, the above variation of $\chi$ vanishes as soon as
$$ \left| \frac{x-y}{\eps} \right| > 1 \qquad \text{ and } \qquad \left| \frac{x-x_0}{\eps} \right| > 1 \,. $$
Since from \eqref{eq:estx-y} we have $|x-x_0| \leq 2|x-y| $, we deduce that
whenever $|x-y|\leq \eps$ or $|x-x_0|\leq \eps $ we have $\eps\geq  |x-x_0|/2$.
This improves \eqref{eq:esttranschi} to
$$ \left| \chi\left( \frac{x-y}{\eps} \right) - \chi \left( \frac{x-x_0}{\eps} \right) \right| \leq \| \nabla \chi\|_\infty \frac{2R}{|x-x_0|} \,. $$
Thus, we estimate
\begin{equation}\begin{array}{l}
	\dsp\hphantom{\leq\ } \left| \int_{|y-x_0| \leq R} \left[ \chi\left( \frac{x-y}{\eps} \right) - \chi\left( \frac{x-x_0}{\eps} \right) \right]  K(x-x_0)  u(y) \, dy \right| \\
\dsp\vphantom{\Biggl|} \leq  \int_{|y-x_0| \leq R}   
\frac{6R}{|x-x_0|}\frac{C_0}{|x-x_0|^N}|u(y)| \, dy\\
\dsp\vphantom{\Biggl|} =  \frac{6 C_0 R }{|x-x_0|^{N+1}}\| u \|_{L^1(\R^N)} \,.
	\label{eq:est2chi}
	\end{array}
\end{equation}
Therefore, \eqref{eq:estKvari} and \eqref{eq:est2chi} yield
$$ T_+(u)(x)\leq \Bigl(6 C_0+2^{N+1}C_1\Bigr)\frac{R}{|x-x_0|^{N+1}}\| u \|_{L^1(\R^N)} \quad\text{for every $x$ such that $|x-x_0|>2R$}. $$
Next we integrate over $|x-x_0| > 2R$, and we obtain
\eqref{eq:estT+1} with $P_1=C_N(C_0+C_1)$.

\medskip

{\sc Step 4. Checking of assumption (iv) of Lemma~\ref{l:interp}.}
Let us fix a nonnegative convolution kernel $\widetilde{\rho}\in C^\infty_c(B_1)$
with $\int_{\R^N} \widetilde \rho =1$. Defining $\widetilde{\Delta}_\eps$ as in \eqref{e:deltaeps}
with $\rho^\nu $ substituted with $\widetilde{\rho}$, the inequality \eqref{e:deltaest} obtained in Step 1
is valid, thus given $u\in L^2$,
\begin{equation}\label{e:redone}
\sup_{\eps > 0} \Bigl|(\widetilde{\Delta}_\eps * u)(x)\Bigr| \leq C_N \big( C_0 + C_1  + \| \widehat K \|_\infty\big) Mu (x) \,.
\end{equation}
Similarly, we can use estimate \eqref{e:forstrongest}, which yields
\begin{equation}
	\sup_{\eps>0}\left|(\widetilde\rho_\eps*Su)(x)\right|\leq C_N M(Su)(x).
	\label{eq:fulltilde}
\end{equation}
But from \eqref{e:split} and \eqref{e:associativity} it follows that
$$
T_+ (u)= \sup_{\eps > 0} \biggl| \left( \chi\left( \frac{\cdot}{\eps} \right) K \right) * u \biggr|
 \leq \sup_{\eps>0} \Bigl| \widetilde \Delta_\eps * u \Bigr|
+\sup_{\eps > 0} \Bigl| \widetilde \rho_\eps * (Su) \Bigr| \,.
$$
Therefore, \eqref{e:redone}, \eqref{eq:fulltilde} and Proposition~\ref{p:classmax} yield
$$ \left\| T_+ (u) \right\|_{L^2(\R^N)}
\leq C_N \big( C_0 + C_1 +  \| \widehat K \|_\infty\big) \| u \|_{L^2(\R^N)} \,, $$
which is precisely assumption (iv) of Lemma~\ref{l:interp}
with $P_2=C_N \big( C_0 + C_1 +  \| \widehat K \|_\infty\big)$.
\end{proof}
\begin{remark}
We remark that in Proposition 2 of Section I.7 of \cite{Steinbig},
an estimate on $T_+$ similar to \eqref{e:weakT} is given in the case
of a discontinuous cutoff function $\chi$, obtaining in fact pointwise bounds
on the truncated singular integral operator.
\end{remark}

\section{Estimate of difference quotients}\label{s:diffq}

This section is devoted to the proof of a generalization of
the estimate (that can be found for instance 
in \cite{Stein}) of the difference quotients of a given function
in terms of the maximal function of its gradient.
\begin{lemma} If $u\in BV(\R^N)$, then there exists 
an $\Leb^N$-negligible set ${\mathcal N} \subset \R^N$ such that
\begin{equation}
	|u(x) - u(y)| \leq C_N|x-y| \Bigl((MDu)(x)+(MDu)(y)\Bigr)
	\qquad\text{ for every $x$, $y \in \R^N \setminus {\mathcal N}$,}
	\label{eq:quotBV}
\end{equation}
where $Du$ stands for the distributional derivative of $u$,
which is a measure here,
and $M$ is the maximal function \eqref{eq:Mf}.
\label{Lemma quotients BV}
\end{lemma}
Our result below shows that it is possible to include
some singular integral operators in the gradient, without
any significant loss in the estimate \eqref{eq:quotBV}.
\begin{propos}[Estimate of difference quotients]\label{p:estdiff}
Let $u \in L^1_\loc(\R^N)$ and assume that for every $j = 1 , \ldots , N$ we have
\begin{equation}\label{e:derivative}
\partial_j u = \sum_{k=1}^m S_{jk} g_{jk} \qquad \text{in $\D'(\R^N)$,}
\end{equation}
where $S_{jk}$ are singular integral operators of fundamental type
in $\R^N$ and $g_{jk}$ are finite measures in $\R^N$, for all $j = 1,\ldots,N$, $k = 1,\ldots,m$,
and where $S_{jk}g_{jk}$ is defined in the distribution sense, as in \eqref{eq:S^D}.
Then there exists a nonnegative function $U \in M^1(\R^N)$ and an $\Leb^N$-negligible set ${\mathcal N} \subset \R^N$ such that
\begin{equation}\label{e:incremest}
|u(x) - u(y)| \leq |x-y| \Bigl(U(x) + U(y)\Bigr)
\qquad\text{ for every $x$, $y \in \R^N \setminus {\mathcal N}$.}
\end{equation}
Moreover, we can take $U$ explicitly given by
\begin{equation}\label{e:exprU}
U = \sum_{j=1}^N \sum_{k = 1}^m M_{\{\Upsilon^{\xi,j} , \; \xi \in {\mathbb S}^{N-1}\}} (S_{jk} g_{jk}) \,,
\end{equation}
where the maximal function relative to a family of kernels
is defined in Definition~\ref{d:rhomax},
the functions $\Upsilon^{\xi,j} \in C^\infty_c (\R^N)$ are explicitly defined
for $\xi\in{\mathbb S}^{N-1}$ and  $j = 1,\ldots,N$ by
\begin{equation}\label{e:ups}
\Upsilon^{\xi , j} (w) = h \left( \frac{\xi}{2} - w \right) w_j,
\end{equation}
and the kernel $h$ is chosen such that
\begin{equation}
	h \in C^\infty_c (\R^N),\qquad
	\int_{\R^N} h(y) \, dy = 1,\qquad
	\spt h \subset B_{1/2}.
	\label{eq:kernelh}
\end{equation}
\end{propos}

\begin{proof}
The property that $U$ defined in \eqref{e:exprU} belongs to $M^1$ is
obtained directly by applying Theorem~\ref{t:mainest} (ii).
Indeed the kernel $\Upsilon^{\xi,j}$
defined in \eqref{e:ups} for $\xi\in{\mathbb S}^{N-1}$
belongs to $C^\infty_c (\R^N)$ and has support in the unit ball
because $\spt h \subset B_{1/2}$ and $|\xi|=1$. Moreover
it is uniformly bounded in $L^1$ and in $H^q$ for any $q\in\N$,
thus with Remark~\ref{Rk Hq} we get that \eqref{e:unifrho}
and \eqref{e:normcb} are satisfied for $\Upsilon^{\xi,j}$.

It remains to prove that \eqref{e:incremest} holds
for $U$ defined by \eqref{e:exprU}.
We define as usual $h_r = \frac{1}{r^N} h \left( \frac{\cdot}{r} \right) $ for all $r>0$,
and since $\int h_r=1$ we can write
\begin{equation}\label{e:diffrew}
u(x) - u(y) = \int_{\R^N} h_r \left( z -\frac{x+y}{2} \right) \big( u(x) - u(z) \big) \, dz
+ \int_{\R^N} h_r \left( z -\frac{x+y}{2} \right) \big( u(z) - u(y) \big) \, dz \,.
\end{equation}
We consider one of the integrals in \eqref{e:diffrew},
we assume that $x\not=y$ and we set $r = |x-y|$.

\medskip

{\sc Step 1. Computation in the smooth case.} In order to justify the following computations,
we assume for the moment that the functions $u$ and $g_{jk}$ are smooth.
Thus we can compute
\begin{equation}\label{e:comp}
\begin{split}
& \int_{\R^N} h_r \left( z -\frac{x+y}{2} \right) \big( u(x) - u(z) \big) \, dz \\
= & - \sum_{j=1}^N \int_{\R^N} \int_0^1 h_r \left( z -\frac{x+y}{2} \right) \partial_j u \big( x + s (z-x) \big) (z_j - x_j) \, ds dz \\
= & \sum_{j=1}^N \int_0^1 \int_{\R^N} h_r \left( x- \frac{x+y}{2} - \frac{w}{s} \right) \partial_j u (x-w) \frac{w_j}{s^{N+1}} \, dw ds \\
= & \sum_{j=1}^N \int_0^1 \frac{r}{s^N} \left[ h_r \left( \frac{x-y}{2} - \frac{w}{s} \right) \frac{w_j}{s r} \right] \stackrel{w}{*} \Big[ \partial_j u (w) \Big] (x) \, ds \,.
\end{split}
\end{equation}
Now, as usual we define for every $\eps>0$
$$ \Upsilon^{\xi , j}_\eps (w) = \frac{1}{\eps^N} \Upsilon^{\xi , j} \left(\frac{w}{\eps}\right) \,. $$
Then from \eqref{e:comp} and the expression for the derivative of $u$ in \eqref{e:derivative}, we deduce that
\begin{equation}\begin{array}{l}
	\dsp \int_{\R^N} h_r \left( z -\frac{x+y}{2} \right) \big( u(x) - u(z) \big) \, dz = r \sum_{j=1}^N \int_0^1 \left[ \Upsilon^{\frac{x-y}{|x-y|} , j}_{s r} * \partial_j u \right] (x) \, ds\\
	\dsp\hphantom{\int_{\R^N} h_r \left( z -\frac{x+y}{2} \right) \big( u(x) - u(z) \big) \, dz}
	=r \sum_{j=1}^N\sum_{k=1}^m \int_0^1 \left[ \Upsilon^{\frac{x-y}{|x-y|} , j}_{s r} * \Bigl(S_{jk}g_{jk}\Bigr) \right] (x) \, ds.
	\label{eq:formdiffups}
	\end{array}
\end{equation}
Therefore
$$\begin{array}{l}
	\dsp\hphantom{\leq\ }\left|\int_{\R^N} h_r \left( z -\frac{x+y}{2} \right) \big( u(x) - u(z) \big) \, dz\right|
	\dsp\leq|x-y|\sum_{j=1}^N\sum_{k=1}^m \int_0^1 \left| \Upsilon^{\frac{x-y}{|x-y|} , j}_{s r} * \Bigl(S_{jk}g_{jk}\Bigr) (x) \right| \, ds \,, \end{array}
$$
and we have the estimate
\begin{equation}\begin{array}{ll}
	\dsp \sum_{j=1}^N\sum_{k=1}^m \int_0^1 
	\left| \Upsilon^{\frac{x-y}{|x-y|} , j}_{s r} * \Bigl(S_{jk}g_{jk}\Bigr) (x) \right| \, ds &
	\dsp \leq\sum_{j=1}^N\sum_{k=1}^m \int_0^1 \sup_{\xi \in {\mathbb S}^{N-1}}
	\sup_{\eps > 0}\left| \Upsilon^{\xi , j}_{\eps} * \Bigl(S_{jk}g_{jk}\Bigr) (x) \right| \, ds\\
	&\dsp = \sum_{j=1}^N\sum_{k=1}^m M_{\{\Upsilon^{\xi,j} , \; \xi \in {\mathbb S}^{N-1}\}}
	\Bigl(S_{jk}g_{jk}\Bigr) (x)
	 = U(x) \, .
	\label{eq:estformdiffups}
\end{array}\end{equation}
Since the second term in \eqref{e:diffrew} gives a similar
contribution with $x$ and $y$ exchanged, this concludes \eqref{e:incremest}
when $u$ and $g_{jk}$ are smooth. 

\medskip

{\sc Step 2. Approximation argument.}
In the general case $u\in L^1_{\loc}$ and $g_{jk}$ measures,
we still have that $U$ defined in \eqref{e:exprU} belongs to $M^1$.
We take a smoothing sequence $\zeta_{1/n}(w)=n^N\zeta(nw)$,
$\zeta\in C^\infty_c$, $\int\zeta=1$, $\spt\zeta\subset B_{\alpha}$
where $\alpha>0$ is such that $\spt h\subset B_{1/2-\alpha}$.
Convolving \eqref{e:derivative} with $\zeta_{1/n}$
yields that $u^n\equiv\zeta_{1/n} * u$ verifies the same
assumption as $u$, with associated functions $g_{jk}^n=\zeta_{1/n} *g_{jk}$.
Since these functions are smooth we can apply the result proved in Step 1.
In particular, \eqref{eq:formdiffups} can be written
\begin{equation}\begin{array}{l}
	\dsp \int_{\R^N} h_r \left( z -\frac{x+y}{2} \right) \big( u_n(x) - u_n(z) \big) \, dz
	=r \sum_{j=1}^N\sum_{k=1}^m \int_0^1 \left[ \Bigl(\zeta_{1/n}*\Upsilon^{\frac{x-y}{|x-y|} , j}_{s r} \Bigr)* \Bigl(S_{jk}g_{jk}\Bigr) \right] (x) \, ds.
	\label{eq:formdiffupsn}
	\end{array}
\end{equation}
Then, for $\eps>0$, $\xi\in{\mathbb S}^{N-1}$, $n\in\N$, we write the kernel in the following way,
$$ \begin{array}{l}
	\dsp \zeta_{1/n}*\Upsilon^{\xi , j}_{\eps}=\Bigl(\zeta*\Upsilon^{\xi , j}_{n\eps}\Bigr)_{1/n}\quad\text{if }n\eps\leq 1,\\
	\dsp \zeta_{1/n}*\Upsilon^{\xi , j}_{\eps}=\Bigl(\zeta_{1/(n\eps)}*\Upsilon^{\xi , j}\Bigr)_\eps\quad\text{if }n\eps>1 \,. \end{array}
$$
Setting
$$	\rho^{n,\eps,\xi,j}=\left\{\begin{array}{l}
	\dsp\zeta*\Upsilon^{\xi , j}_{n\eps}\quad\text{if }n\eps\leq 1,\\
	\dsp\zeta_{1/(n\eps)}*\Upsilon^{\xi , j}\quad\text{if }n\eps>1,
	\end{array}\right.
$$
we have $\spt \rho^{n,\eps,\xi,j}\subset B_1$,
thus by applying Theorem~\ref{t:mainest} (ii) with $\nu=(n,\eps,\xi)$
we get that
$$	\overline U\equiv\sum_{j=1}^N\sum_{k=1}^m\sup_{n,\eps,\xi}\Bigl|\Bigl(\zeta_{1/n}*\Upsilon^{\xi , j}_{\eps}\Bigr)*\Bigl(S_{jk}g_{jk}\Bigr)\Bigr|\in M^1 \,.
$$
In particular, there exists an $\Leb^N$-negligible set ${\mathcal N}_1 \subset \R^N$
such that $\overline U(x)<\infty$ for all $x\not\in{\mathcal N}_1$.
We have also $u^n\rightarrow u$ in $L^1_{\loc}$, thus after extraction
of a subsequence, $u^n\rightarrow u$ a.e, and there exists
an $\Leb^N$-negligible set ${\mathcal N}_2 \subset \R^N$ such that
$u^n(x)\rightarrow u(x)$ for all $x\not\in{\mathcal N}_2$.
Setting ${\mathcal N}={\mathcal N}_1\cup{\mathcal N}_2$,
we can pass to the limit in \eqref{eq:formdiffupsn} by dominated convergence
in $s$ for all $x\not\in{\mathcal N}$. This yields \eqref{eq:formdiffups}.
Finally, the estimate \eqref{eq:estformdiffups} is still valid,
proving the result.
\end{proof}

\begin{remark}
The case $u \in BV$ previously stated in Lemma~\ref{Lemma quotients BV} is obtained 
from Proposition~\ref{p:estdiff} by taking $m=1$ and $S_{jk} = \delta_0$.
Noticing that according to \eqref{e:forstrongest} we have
$M_{\{\Upsilon^{\xi,j}\}}(g_j)\leq C_N Mg_j$, we recover estimate \eqref{eq:quotBV}.
\end{remark}

\section{Regular Lagrangian flow: definition and main estimate}\label{s:estimate}

We start by recalling the notion of convergence in measure for measurable functions and by introducing some related notation.

\begin{definition}[Convergence in measure] We say that a sequence of measurable functions
$u_n : \R^N \to \R$ {\em converges locally in measure in $\R^N$} towards a measurable function
$u : \R^N \to \R$ if for every $\gamma > 0$ and every $r>0$ there holds
$$ \Leb^N  \big( \big\{ x \in B_r \; : \; |u_n(x) - u(x)| > \gamma \big\} \big) \to 0 
\qquad \text{as $n \to +\infty$.}
$$
We simply say that $u_n$ converges towards $u$ in measure in $\R^N$ when the same is true
with $B_r$ replaced by the whole space $\R^N$.

We denote by $L^0(\R^N)$ the space of real-valued measurable functions on $\R^N$, defined a.e.~with respect to the Lebesgue measure $\Leb^N$, endowed with the convergence in measure. We denote by $L^0_\loc(\R^N)$ the same space, when we mean that it is endowed with the local convergence in measure.
We shall also denote by ${\mathcal B}(E,F)$ the space of bounded functions between the sets $E$ and $F$.
\end{definition}

We shall always consider in the following vector fields $b : (0,T) \times \R^N \to \R^N$ satisfying
at least the following growth condition:
\begin{itemize}
\item[(R1)] The vector field $b$ can be decomposed as
\begin{equation}\label{e:grass2}
\frac{b(s,x)}{1+|x|} = \tilde b_1(s,x) + \tilde b_2 (s,x) \,,
\end{equation}
with
\begin{equation}
	\tilde b_1 \in L^1 \big( (0,T);L^1(\R^N) \big),\qquad
	\tilde b_2 \in L^1\big( (0,T) ; L^\infty(\R^N) \big).
	\label{e:grass2space}
\end{equation}
\end{itemize}

The following is the by now usual definition of flow for the ordinary differential equation in the case of non smooth vector fields.
In all the following, we denote by $\log L(\R^N)$ the space of measurable functions $u : \R^N \to \R$ such that
$\int_{\R^N} \log \big( 1 + |u(x)| \big) \, dx$ is finite, and the local space $\log L_\loc (\R^N)$ is defined accordingly.

\begin{definition}[Regular Lagrangian flow]\label{d:rlf}
Let $b : (0,T) \times \R^N \to \R^N$ be a vector field satisfying assumption (R1) and fix $t \in [0,T)$. We say that a map 
\begin{equation}\label{e:rlfclass}
X \in C\big( [t,T]_s ; L^0_\loc(\R^N_x) \big) \cap
{\mathcal B}\big( [t,T]_s ; \log L_\loc(\R^N_x) \big)
\end{equation}
is a {\em regular Lagrangian flow (in the renormalized sense)} relative to $b$ starting at time $t$ if the following properties hold:
\begin{itemize}
\item[(i)] The equation 
\begin{equation}\label{e:rlfeq}
\partial_s \Bigl(\beta\big(X(s,x)\big)\Bigr) = \beta' \big( X(s,x) \big)b \big(s,X(s,x)\big) 
\end{equation}
holds in $\D' \big( (t,T) \times \R^N \big)$, for every function $\beta \in C^1(\R^N;\R)$ satisfying 
$$ |\beta(z)| \leq C \big( 1 + \log (1+|z|) \big) \quad \text{and} \quad
|\beta'(z)| \leq \frac{C}{1+|z|} \qquad
\text{for all $z \in \R^N$,}
$$
for some constant $C$;
\item[(ii)] $X(t,x) = x$ for $\Leb^N$-a.e.~$x \in \R^N$;
\item[(iii)] There exists a constant $L \geq 0$ such that for every $s \in [t,T]$ there holds 
\begin{equation}\label{e:defcompr}
X(s,\cdot)_\# \Leb^N \leq L \Leb^N \,,
\end{equation}
i.e.
\begin{equation}
	\Leb^N \left( \left\{ x \in \R^N \; : \; X(s,x) \in B \right\} \right) \leq L \Leb^N (B) \qquad \text{ for every Borel set $B \subset \R^N$.}
	\label{eq:absmeas}
\end{equation}
\end{itemize}
The constant $L$ in \eqref{e:defcompr} is called the {\em compression constant} of the flow $X$.
The condition \eqref{e:defcompr} can also be formulated as
\begin{equation}
	\int_{\R^N}\varphi(X(s,x))\,dx\leq L\int_{\R^N}\varphi(x)\,dx
	\qquad\mbox{for all measurable }\varphi:\R^N\rightarrow[0,\infty).
	\label{e:comprfunct}
\end{equation}
Note that (iii) and \eqref{e:grass2}-\eqref{e:grass2space} enable to get a right-hand side of \eqref{e:rlfeq} in $L^1((t,T),L_{\loc}^1(\R^N))$.
\end{definition}

\begin{remark} (i) In the case of smooth flows, the bounded compression condition \eqref{e:defcompr} corresponds to the lower bound $JX(s,x) \geq 1/L >0$ on the Jacobian of the flow. (ii)
In the sequel we shall also need to make explicit the dependence of the flow on the initial time
$t \in [0,T]$ chosen in Definition~\ref{d:rlf}, thus we shall use the notation $X(s,t,x)$ to indicate
the value at time $s$ of the regular Lagrangian flow relative to $b$ and starting at time $t$,
with $0\leq t\leq s\leq T$.
\end{remark}

Our definition of regular Lagrangian flow slightly differs from the usual one (see for instance \cite{A,notes,notes2}),
since in general we do not assume global boundedness of the vector field $b$,
and thus we do not have that $X(s,\cdot)$ is locally integrable in $\R^N$.
For this reason we state the equation in Definition~\ref{d:rlf}(i)
in the renormalized sense, since the usual distributional equation is a priori meaningless.
A bound on $X(s,x)$ in $\log L_\loc(\R^N)$
(as stated in \eqref{e:rlfclass}) indeed follows from the integration in $s$ of \eqref{e:rlfeq},
knowing that the right-hand side is in $L^1((t,T),L_{\loc}^1(\R^N))$, see \eqref{eq:logpoint} below.
This estimate can be slightly strenghtened by the a priori estimate on the growth of the flow stated in the following lemma.
We preliminarily introduce the notation for the sublevels.

\begin{definition}[Sublevels]\label{d:sublevel}
Let $X : [t,T] \times \R^N \to \R^N$ be a measurable map. For every $\lambda > 0$ we define the sublevel
$$	G_\lambda = \Big\{ x \in \R^N \; : \; |X(s,x)| \leq \lambda \text{ for almost all $s\in [t,T]$} \Big\}\,.
$$
\end{definition}

\begin{lemma}[Estimate of the superlevels]\label{l:superlevels} 
Let $b : (0,T) \times \R^N \to \R^N$ be a vector field satisfying assumption (R1)
and let $X : [t,T] \times \R^N \to \R^N$ be a regular Lagrangian flow relative to $b$ starting at time $t$, with compression constant $L$.
Then for $B_r$ the ball of radius $r$ centered at the origin,
\begin{equation}
	\mbox{for all }s\in[t,T],\ \int_{B_r}\log\left(\frac{1+|X(s,x)|}{1+r}\right)\,dx
	\leq L \| \tilde b_1 \|_{L^1((0,T);L^1(\R^N))} + \Leb^N (B_r)\| \tilde b_2 \|_{L^1((0,T);L^\infty(\R^N))}
	\label{eq:logpoint}
\end{equation}
and
\begin{equation}\label{e:loggrowth}
\int_{B_r} \mathop{\rm ess\, sup}_{t \leq s \leq T}\, \log \left( \frac{ 1 + |X(s,x)|}{1+r} \right) \, dx \leq 
L \| \tilde b_1 \|_{L^1((0,T);L^1(\R^N))} + \Leb^N (B_r)\| \tilde b_2 \|_{L^1((0,T);L^\infty(\R^N))} \,,
\end{equation}
where $\tilde b_1$ and $\tilde b_2$ are as in \eqref{e:grass2}-\eqref{e:grass2space}.
This in particular implies that for every regular Lagrangian flow $X$ relative to $b$ starting at time $t$
 with compression constant $L$, we have for all $r,\lambda>0$
\begin{equation}
	\Leb^N \big( B_r \setminus G_\lambda \big) \leq g(r,\lambda) \,,
	\label{eq:estsuperlevels}
\end{equation}
where the function $g$ only depends on $\| \tilde b_1 \|_{L^1((0,T);L^1(\R^N))}$, $\| \tilde b_2 \|_{L^1((0,T);L^\infty(\R^N))}$ and $L$, and satisfies $g(r,\lambda) \downarrow 0$ for $r$ fixed and $\lambda \uparrow +\infty$.
\end{lemma}
\begin{proof}The proof can be found for instance in \cite[Proposition 3.2]{estimates}, but for
completeness we write it in full details.
Let us take $\beta_\varepsilon(z)=\log(1+\sqrt{|z|^2+\varepsilon^2})$, for some $\varepsilon>0$.
Then \eqref{e:rlfeq}, (R1) and \eqref{e:comprfunct} yields that
$\partial_s\bigl(\beta_\varepsilon(X(s,x))\bigr)\in L^1((t,T),L^1_\loc(\R^N))$.
Thus, for almost all $x\in\R^N$, $\partial_s\bigl(\beta_\varepsilon(X(\cdot,x))\bigr)\in L^1(t,T)$.
This implies that for such $x$, $\beta_\varepsilon(X(\cdot,x))$ coincides almost everywhere with an absolutely continuous function
$\Xi_\varepsilon(\cdot,x)$ in $[t,T]$. We have for all $s\in[t,T|$
$$ \Xi_\varepsilon(s,x)=\Xi_\varepsilon(t,x)+\int_t^s{\beta_\varepsilon}'\bigl(X(\tau,x)\bigr)b\bigl(\tau,X(\tau,x)\bigr)d\tau \,,
$$
thus
\begin{equation}
	\mbox{for a.e. }x\in\R^N,\quad\mbox{for a.e. }s\in (t,T),\qquad
	\beta_\varepsilon(X(s,x))=\Xi_\varepsilon(t,x)+\int_t^s{\beta_\varepsilon}'\bigl(X(\tau,x)\bigr)b\bigl(\tau,X(\tau,x)\bigr)d\tau.
	\label{eq:intcharae}
\end{equation}
But since the integral in the right-hand side belongs to $C([t,T],L^1_\loc)$, and
$\beta_\varepsilon(X)\in C\big( [t,T] ; L^0_\loc \big)\cap{\mathcal B}([t,T],L^1_\loc)$,
we have that $\Xi_\varepsilon(t,\cdot)\in L^1_\loc$, and \eqref{eq:intcharae} is valid for all $s$, i.e.
\begin{equation}
	\mbox{for all }s\in [t,T],\quad\mbox{for a.e. }x\in\R^N,\qquad
	\beta_\varepsilon(X(s,x))=\beta_\varepsilon(X(t,x))+\int_t^s{\beta_\varepsilon}'\bigl(X(\tau,x)\bigr)b\bigl(\tau,X(\tau,x)\bigr)d\tau.
	\label{eq:intcharall}
\end{equation}
We have
$$ {\beta_\varepsilon}'(z)=\frac{1}{1+\sqrt{|z|^2+\varepsilon^2}}\frac{z}{\sqrt{|z|^2+\varepsilon^2}}\,,\qquad
	\left|{\beta_\varepsilon}'(z)\right|\leq\frac{1}{1+|z|} \,,
$$
thus for all $s\in[t,T]$ and $a.e.\ x\in\R^N$,
$$ \beta_\varepsilon(X(s,x))\leq\beta_\varepsilon(X(t,x))
	+\int_t^s\frac{\left|b\bigl(\tau,X(\tau,x)\bigr)\right|}{1+|X(\tau,x)|}\,d\tau \,.
$$
Letting $\varepsilon\rightarrow 0$ this yields that for all $s\in[t,T]$ and $a.e.\ x\in\R^N$,
\begin{equation}
	\log(1+|X(s,x)|)\leq\log(1+|x|)
	+\int_t^s\frac{\left|b\bigl(\tau,X(\tau,x)\bigr)\right|}{1+|X(\tau,x)|}\,d\tau.
	\label{eq:intcharall3}
\end{equation}
In particular, integrating this over $x\in B_r$ yields \eqref{eq:logpoint}.
Similarly, weakening \eqref{eq:intcharall3} to $a.e.\ x\in\R^N$, $a.e.\ s\in(t,T)$, yields \eqref{e:loggrowth}.
Then, we have
$$ \int_{B_r} \mathop{\rm ess\, sup}_{t \leq s \leq T}\, \log \left( \frac{ 1 + |X(s,x)|}{1+r} \right) \, dx
	\geq \Leb^N(B_r\backslash G_\lambda)\log\left(  1 + \lambda \right)-\Leb^N(B_r)\log(1+r) \,,
$$
proving \eqref{eq:estsuperlevels} with
$$ g(r,\lambda)=\frac{L \| \tilde b_1 \|_{L^1((0,T);L^1(\R^N))} + \Leb^N (B_r)\| \tilde b_2 \|_{L^1((0,T);L^\infty(\R^N))}
	+\Leb^N(B_r)\log(1+r)}{\log(1+\lambda)} \,. $$
\end{proof}
\begin{remark}\label{r:Xsmallinfinity}
The inequality \eqref{eq:intcharall3} can be reversed, by estimating \eqref{eq:intcharall} from
the other side. Thus we have for all $s\in[t,T]$ and $a.e.\ x\in\R^N$,
$$
	\log(1+|x|)\leq\log(1+|X(s,x)|)
	+\int_t^s\frac{\left|b\bigl(\tau,X(\tau,x)\bigr)\right|}{1+|X(\tau,x)|}\,d\tau.
$$
In particular, given $\lambda>0$ and $r>0$, the points $x$ where $|X(s,x)|\leq\lambda$
and $|x|>r$ satisfy
$$ \int_t^s\left|\tilde b_1\bigl(\tau,X(\tau,x)\bigr)\right|\,d\tau
	\geq \log(1+r)-\log(1+\lambda)-\| \tilde b_2 \|_{L^1((0,T);L^\infty(\R^N))} \,.
$$
In particular, the following counterpart of \eqref{eq:estsuperlevels} holds: for fixed $\lambda$, we have
$$ \Leb^N\left(\left\{x\in\R^N \, : \, |X(s,x)|\leq\lambda, |x|>r\right\}\right)
	\longrightarrow 0 \qquad\mbox{as }r\rightarrow\infty \,,
$$
uniformly for $s\in[t,T]$.
\end{remark}

Let us now recall some classical properties related to the notion of equi-integrability.

\begin{definition}[Equi-integrability] Let $\Omega$ be an open subset of $\R^N$. We say that a bounded family $\{ \varphi_i \}_{i \in I} \subset L^1(\Omega)$ is equi-integrable if the following two conditions hold:
\begin{itemize}
\item[(i)] For any $\eps > 0$ there exists a Borel set $A \subset \Omega$ with finite measure such that $\int_{\Omega \setminus A} |\varphi_i| \, dx \leq \eps$ for any $i \in I$;
\item[(ii)] For any $\eps > 0$ there exists $\delta > 0$ such that, for every Borel set $E \subset \Omega$ with with $\Leb^N(E) \leq \delta$, there holds $\int_E | \varphi_i | \, dx \leq \eps$ for any $i \in I$.
\end{itemize}
\end{definition}
We recall in passing that the Dunford-Pettis theorem ensures that a bounded family in $L^1(\Omega)$
is relatively compact for the weak $L^1$ topology if and only if it is equi-integrable. Also notice that every finite family is (trivially) equi-integrable.
An interesting property is that a sequence $u_n\in L^1(\R^N)$ converges to $u$ in $L^1(\R^N)$
if and only if it is equi-integrable and $u_n$ converges to $u$ locally in measure.
The following lemma can be proved with elementary tools.

\begin{lemma}\label{l:equiint}
Consider a family $\{ \varphi_i \}_{i \in I} \subset L^1(\Omega)$ which is bounded in $L^1(\Omega)$.
Then this family is equi-integrable if and only if for every $\eps > 0$, there exists a constant $C_\eps$
and a Borel set $A_\eps\subset\Omega$ with finite measure such that
such that for every $i \in I$ we can write
$$
\varphi_i = \varphi_i^1 + \varphi_i^2 \,,
$$
with
$$
\| \varphi_i^1 \|_{L^1(\Omega)} \leq \eps \qquad \text{ and } \qquad
\spt(\varphi_i^2)\subset A_\eps,\quad
\| \varphi_i^2 \|_{L^2(\Omega)} \leq C_\eps \qquad \text{ for all $i \in I$.}
$$
\end{lemma}

In order to obtain results of well-posedness (i.e., existence, uniqueness and stability)
for the regular Lagrangian flow, the mere growth conditions in (R1) are not sufficient.
Some assumptions on the space derivatives of the vector field are needed
(see for instance the discussion in \cite{notes, notes2, thesis}).
We are interested in the case when the space derivatives of the vector field can be expressed
as singular integrals of $L^1$ functions, with singular kernels of fundamental type
as in Definition~\ref{d:fundker}. We thus assume that in addition to (R1) $b$ satisfies the following assumption.

\begin{itemize}
\item[(R2)] For every $i$, $j=1,\ldots,N$ we have
\begin{equation}\label{e:exprder}
\partial_j b^i = \sum_{k=1}^m S_{jk}^i g_{jk}^i \qquad \text{in $\D'\big((0,T) \times \R^N\big)$,}
\end{equation}
where $S_{jk}^i$ are singular integral operators of fundamental type in $\R^N$
(acting as operators in $\R^N$ independently of time) and the functions  $g_{jk}^i \in L^1 \big( (0,T) \times \R^N \big)$ for every $i$, $j=1,\ldots,N$ and every $k=1,\ldots,m$. For notational simplicity, we shall also use the vectorial notation
\begin{equation}
	\partial_j b = \sum_{k=1}^m S_{jk} g_{jk} \qquad \text{in $\D' \big( (0,T) \times \R^N \big)$,}
	\label{e:exprdervect}
\end{equation}
in which $S_{jk}$ is a vector consisting of $N$ singular integral operators,
and for every $j=1,\ldots,N$ and every $k=1,\ldots,m$ we have $g_{jk} \in L^1 \big( (0,T) \times \R^N ; \R^N \big)$.
Note that the equations \eqref{e:exprder} or \eqref{e:exprdervect} can also be formulated
as equations in $\D'(\R^N)$ for almost every $t\in(0,T)$.
\end{itemize}

\noindent We also make the following local integrability assumption,
\begin{itemize}
\item[(R3)] The vector field $b$ satisfies
\begin{equation}\label{e:grass1}
b \in L^p_\loc( [0,T]\times \R^N)\quad\mbox{for some }p>1.
\end{equation}
\end{itemize}

We are now in position of proving our main quantitative estimate on the regular Lagrangian flows.
The idea is to consider an integral functional that controls the distance between two flows,
and to derive estimates in the same spirit of \cite{estimates}.
A key tool is given by Theorem~\ref{t:mainest}, that allows to estimate the composition
of singular integral operators with smooth maximal functions, that appears
in \eqref{e:exprU} when estimating the differential quotients as \eqref{e:incremest} in Proposition~\ref{p:estdiff}.

\begin{propos}[Fundamental estimate for flows]\label{l:main}
Let $b$ and $\bar b$ be two vector fields satisfying assumption (R1),
and assume that $b$ also satisfies assumptions (R2), (R3).
Fix $t \in [0,T)$ and let $X$ and $\bar X$ be regular Lagrangian flows starting at time $t$ associated to $b$ and $\bar b$ respectively, with compression constants $L$ and $\bar L$. Then the following holds.
For every $\gamma > 0$ and $r > 0$ and for every $\eta > 0$ there exist $\lambda > 0$ and $C_{\gamma,r,\eta}>0$ such that
\begin{equation}\label{e:mainlemma}
\Leb^N \Big( B_r \cap \big\{ |X(s,\cdot) - \bar X (s,\cdot) | > \gamma \big\} \Big) \leq
C_{\gamma,r,\eta} \| b - \bar b \|_{L^1((0,T) \times B_\lambda)} + \eta
\quad\mbox{ for all }s \in [t,T].
\end{equation}
The constants $\lambda > 0$ and $C_{\gamma,r,\eta}>0$, beside depending on $\gamma$, $r$ and $\eta$, also depend on
\begin{itemize}
\item The equi-integrability in $L^1\big( (0,T) \times \R^N \big)$ of the functions $g_{jk}$ associated to $b$ as in \eqref{e:exprdervect};
\item The norms of the singular integral operators $S_{jk}$ associated to $b$ as in \eqref{e:exprdervect}
(i.e., the constants $C_0 + C_1+\|\widehat K\|_\infty$ from Definition~\ref{d:fundker});
\item The norm $\|b\|_{L^p((0,T) \times B_\lambda)}$ corresponding to \eqref{e:grass1};
\item The quantities $\| \tilde b_1 \|_{L^1(L^1)} + \| \tilde b_2 \|_{L^1(L^\infty)}$ and 
$\| \tilde{\bar b}_1 \|_{L^1(L^1)} + \| \tilde{\bar b}_2 \|_{L^1(L^\infty)}$, for decompositions of $b$ and $\bar b$ as in \eqref{e:grass2}-\eqref{e:grass2space};
\item The compression constants $L$ and $\bar L$.
\end{itemize}
\end{propos}

\begin{proof}
For any $\delta>0$, $\lambda>0$ and $s\in[t,T]$ let us define the quantity
\begin{equation}\label{e:functional}
\Phi_\delta(s) = \int_{B_r \cap G_\lambda \cap \overline G_\lambda} \log \left( 1+ \frac{|X(s,x) - \overline X (s,x)|}{\delta} \right) \, dx \,,
\end{equation}
where $G_\lambda$ and $\overline G_\lambda$ are the sublevels of $X$ and $\overline X$ respectively, defined as in Definition~\ref{d:sublevel}.
Because of the continuity statement in \eqref{e:rlfclass}, $\Phi_\delta$ is continuous, and since
in \eqref{e:functional} the values of $x$ that are involved correspond to bounded trajectories,
we are able to use \eqref{e:rlfeq} in the classical sense of derivatives of absolutely
continuous functions. Thus  $\Phi_\delta$ is also absolutely continuous, with
\begin{equation*}
\begin{split}
\Phi'_\delta(s) \,=\, & \int_{B_r \cap G_\lambda \cap \overline G_\lambda} \frac{b(s,X(s,x)) - \overline b (s,\overline X(s,x))}{\delta + |X(s,x) - \overline X(s,x)|}\cdot\frac{X(s,x) - \overline X(s,x)}{|X(s,x) - \overline X(s,x)|} \, dx\\
\leq \, & \int_{B_r \cap G_\lambda \cap \overline G_\lambda} 
\frac{|b(s,X(s,x)) - \overline b (s,\overline X(s,x))|}{\delta + |X(s,x) - \overline X(s,x)|} \, dx \\
\leq \, & \int_{B_r \cap G_\lambda \cap \overline G_\lambda} 
\frac{|b(s,\overline X(s,x)) - \overline b (s,\overline X(s,x))|}{\delta + |X(s,x) - \overline X(s,x)|} \, dx + \int_{B_r \cap G_\lambda \cap \overline G_\lambda}  
\frac{|b(s,X(s,x)) -  b (s,\overline X(s,x))|}{\delta + |X(s,x) - \overline X(s,x)|} \, dx \\
\leq \, & \frac{\overline L}{\delta} \int_{B_\lambda} |b(s,x) - \overline b(s,x)| \, dx \\
& +  \int_{B_r \cap G_\lambda \cap \overline G_\lambda}  
\min \left\{ \frac{|b(s,X(s,x))|+|b(s,\overline X(s,x))|}{\delta} \; ; \;  
\frac{|b(s,X(s,x)) - b(s,\overline X (s,x))|}{|X(s,x) - \overline X (s,x) |} \right\}  dx \,.
\end{split}
\end{equation*}
We now apply Proposition~\ref{p:estdiff} for almost all $s$, which gives the existence of a function $U(s) \in M^1(\R^N)$ estimating the difference quotients of the vector field $b$. We obtain
\begin{equation}\label{e:unog}
\begin{split}
\Phi'_\delta(s) \leq \, & \frac{\overline L}{\delta} \| b(s,\cdot) - \overline b(s,\cdot) \|_{L^1(B_\lambda)}  \\
& +  \int_{B_r \cap G_\lambda \cap \overline G_\lambda}  
\min \left\{ \frac{|b(s,X(s,x))|}{\delta} + \frac{|b(s,\overline X(s,x))|}{\delta} \; ; \;  
U(s,X(s,x)) + U(s,\overline X (s,x)) \right\}  dx \,. 
\end{split}
\end{equation}
Now, observing that $\Phi_\delta(t)=0$, for any $\tau\in[t,T]$ we integrate \eqref{e:unog}
over $s\in(t,\tau)$ to get
\begin{equation}\label{e:unogint}
\begin{split}
\Phi_\delta(\tau) \leq \, & \frac{\overline L}{\delta} \| b - \overline b \|_{L^1((t,\tau)\times B_\lambda)}  \\
& \mkern -20mu+  \int_t^\tau\!\!\!\int_{B_r \cap G_\lambda \cap \overline G_\lambda}  
\min \left\{ \frac{|b(s,X(s,x))|}{\delta} + \frac{|b(s,\overline X(s,x))|}{\delta} \; ; \;  
U(s,X(s,x)) + U(s,\overline X (s,x)) \right\}  dx \,ds\,. 
\end{split}
\end{equation}
Recall that $U$ is given by \eqref{e:exprU}. Let us fix $\eps > 0$, that will be chosen later,
and apply Lemma~\ref{l:equiint} to the finite family of functions $g_{jk} \in L^1 \big( (0,T)\times \R^N \big)$.
This gives the existence of a constant $C_\eps$ and a set $A_\eps$ with finite measure such that
for every $j=1,\ldots,N$ and $k=1,\ldots,m$, we have a decomposition
$$
g_{jk}(s,x) = g^1_{jk}(s,x) + g^2_{jk}(s,x)
$$
satisfying
\begin{equation}\label{e:g12eps}
\| g^1_{jk}\|_{L^1((0,T)\times \R^N)} \leq \eps 
\qquad\text{ and }\qquad
\spt(g^2_{jk})\subset A_\eps,\quad
\| g^2_{jk}\|_{L^2((0,T)\times \R^N)} \leq C_\eps \,.
\end{equation}
The constant $C_\eps$ measures the equi-integrability of the family $g_{jk}$ in $ L^1 \big( (0,T)\times \R^N \big)$.
We deduce that 
\begin{equation}\label{e:decg12}
\begin{split}
U \, = \, & \sum_{k = 1}^m \sum_{j=1}^N M_{\{\Upsilon^{\xi,j} , \; \xi \in {\mathbb S}^{N-1}\}} (S_{jk} g_{jk}) \\
\leq \, & \sum_{k = 1}^m \sum_{j=1}^N M_{\{\Upsilon^{\xi,j} , \; \xi \in {\mathbb S}^{N-1}\}} (S_{jk} g^1_{jk})
+ \sum_{k = 1}^m \sum_{j=1}^N M_{\{\Upsilon^{\xi,j} , \; \xi \in {\mathbb S}^{N-1}\}} (S_{jk} g^2_{jk}) \\
\equiv\,& U^1+U^2.
\end{split}
\end{equation}
Plugging \eqref{e:decg12} into \eqref{e:unogint} gives
\begin{equation}\label{e:gdue}
\begin{split}
\Phi_\delta(\tau) \leq \, & \frac{\overline L}{\delta} \| b - \overline b \|_{L^1((t,\tau)\times B_\lambda)}  \\
& \mkern -20mu+  \int_t^\tau\!\!\!\int_{B_r \cap G_\lambda \cap \overline G_\lambda}  
\mkern-10mu\min \left\{ \frac{|b(s,X(s,x))|}{\delta} + \frac{|b(s,\overline X(s,x))|}{\delta} \; ; \;  
U^1(s,X(s,x)) + U^1(s,\overline X (s,x)) \right\}  dx \,ds\\
& \mkern -20mu+  \int_t^\tau\!\!\!\int_{B_r \cap G_\lambda \cap \overline G_\lambda}  
\mkern-10mu\min \left\{ \frac{|b(s,X(s,x))|}{\delta} + \frac{|b(s,\overline X(s,x))|}{\delta} \; ; \;  
U^2(s,X(s,x)) + U^2(s,\overline X (s,x)) \right\}  dx \,ds.
\end{split}
\end{equation}
For the second double integral we write
\begin{equation*}
\begin{split}
&\int_{B_r \cap G_\lambda \cap \overline G_\lambda}  
\mkern-10mu\min \left\{ \frac{|b(s,X(s,x))|}{\delta} + \frac{|b(s,\overline X(s,x))|}{\delta} \; ; \;  
U^2(s,X(s,x)) + U^2(s,\overline X (s,x)) \right\}  dx\\
\leq\,& \int_{B_r \cap G_\lambda \cap \overline G_\lambda}  
\mkern-10mu \biggl(U^2(s,X(s,x)) + U^2(s,\overline X (s,x))\biggr)\,dx\\
\leq\,&(L+\overline L)\int_{B_\lambda}U^2(s,x)dx.
\end{split}
\end{equation*}
Therefore, the second double integral $I_2$ in \eqref{e:gdue} is estimated by
\begin{equation}
	I_2\leq(L+\overline L)\|U^2\|_{L^1((t,\tau)\times B_\lambda)}
	\leq (L+\overline L)\left[(\tau-t) \Leb^N(B_\lambda)\right]^{1/2}\|U^2\|_{L^2((t,\tau)\times \R^N)}.
	\label{e:estI2}
\end{equation}
Now, applying Theorem~\ref{t:mainest} to the operator
\begin{equation}
	g\mapsto \sum_{k = 1}^m \sum_{j=1}^N M_{\{\Upsilon^{\xi,j} , \; \xi \in {\mathbb S}^{N-1}\}} (S_{jk} g_{jk}),
	\label{eq:opmaxsing}
\end{equation}
where the kernels $\Upsilon^{\xi,j}$ are defined in \eqref{e:ups}-\eqref{eq:kernelh}, yields
that this operator \eqref{eq:opmaxsing} is bounded $L^2(\R^N)\rightarrow L^2(\R^N)$
and $L^1(\R^N)\rightarrow M^1(\R^N)$, with constants $P_2$ and $P_1$ respectively, depending
only on the norms of the singular integral operators $S_{jk}$ (recall \eqref{e:care} that
enables to control  a finite sum in $M^1$).
Then, using the simple inequality
$$
	\lpn u(t,x) \rpn_{M^1_{t,x}} \leq \big\| \lpn u(t,x) \rpn_{M^1_x} \big\|_{L^1_t} \,,
$$
we obtain that
\begin{equation*}
\begin{split}
\lpn U^1 \rpn_{M^1((t,\tau) \times \R^N)}
\, \leq \, & P_1 \| g^1 \|_{L^1((t,\tau) \times \R^N)}\,, \\
\left\| U^2 \right\|_{L^2((t,\tau) \times \R^N)}
\, \leq \, & P_2 \| g^2 \|_{L^2((t,\tau) \times \R^N)}\,.
\end{split}
\end{equation*}
This last inequality yields with \eqref{e:estI2} that
\begin{equation}
	I_2\leq (L+\overline L)P_2\left[(\tau-t) \Leb^N(B_\lambda)\right]^{1/2}\| g^2 \|_{L^2((t,\tau) \times \R^N)}.
	\label{e:I_2final}
\end{equation}
Next, we would like to estimate the first double integral $I_1$ in \eqref{e:gdue}.
We observe that
$$
	\lpn U^1(s,X(s,x))\rpn_{M^1((t,\tau)\times (B_r \cap G_\lambda \cap \overline G_\lambda))}
	\leq L\lpn U^1(s,x)\rpn_{M^1((t,\tau)\times B_\lambda)}\,,
$$
and similarly for $\overline X$. Thus denoting by
$$
	\varphi(s,x)=\min \left\{ \frac{|b(s,X(s,x))|}{\delta} + \frac{|b(s,\overline X(s,x))|}{\delta} \; ; \;  
U^1(s,X(s,x)) + U^1(s,\overline X (s,x)) \right\} \,,
$$
we have
\begin{equation}
	\lpn\varphi\rpn_{M^1((t,\tau)\times (B_r \cap G_\lambda \cap \overline G_\lambda))}
	\leq 2(L+\overline L)\lpn U^1\rpn_{M^1((t,\tau)\times B_\lambda)}
	\leq2(L+\overline L)P_1\| g^1 \|_{L^1((t,\tau) \times \R^N)}.
	\label{e:M1varphi}
\end{equation}
But on the other side,
\begin{equation}
\begin{split}
	\|\varphi\|_{L^p((t,\tau)\times (B_r \cap G_\lambda \cap \overline G_\lambda))}
	\leq\, & \frac{1}{\delta}\left\||b(s,X(s,x))| + |b(s,\overline X(s,x))|\right\|_{L^p((t,\tau)\times (B_r \cap G_\lambda \cap \overline G_\lambda))}\\
	\leq\, &\frac{L^{1/p}+{\overline L}^{1/p}}{\delta}\|b\|_{L^p((t,\tau) \times B_\lambda)}
	\leq 2\frac{(L+\overline L)^{1/p}}{\delta}\|b\|_{L^p((t,\tau) \times B_\lambda)}.
\end{split}
	\label{eq:Lpvarphi}
\end{equation}
Apply now the interpolation Lemma~\ref{Lemma InterpM1} to the function $\varphi$
and using \eqref{e:M1varphi}, \eqref{eq:Lpvarphi} gives
\begin{equation}
\begin{split}
	I_1=&\,\|\varphi\|_{L^1((t,\tau)\times (B_r \cap G_\lambda \cap \overline G_\lambda))}\\
	\leq&\,\frac{p}{p-1}2(L+\overline L)P_1\| g^1 \|_{L^1((t,\tau) \times \R^N)}\\
	&\,\mkern 30mu
	\times\Biggl[1+\log^+\left(\frac{(L+\overline L)^{1/p}\|b\|_{L^p((t,\tau) \times B_\lambda)}}
	{(L+\overline L)P_1\| g^1 \|_{L^1((t,\tau) \times \R^N)}}
	\frac{\left[(\tau-t)\Leb^N(B_r)\right]^{1-1/p}}{\delta}\right)\Biggr],
\end{split}
	\label{eq:intervarphi}
\end{equation}
where we used that the map $z \mapsto z ( 1 + \log^+ (K/z))$ is nondecreasing over $[0,\infty)$.
Plugging the estimates \eqref{eq:intervarphi}, \eqref{e:I_2final} in \eqref{e:gdue},
we deduce with \eqref{e:g12eps} that
\begin{equation}\label{e:integrata}
\begin{split}
\Phi_\delta(\tau) \, \leq \, &  \frac{\overline L}{\delta} \| b  - \overline b \|_{L^1((t,\tau) \times B_\lambda)}
+ \big( L + \overline L \big) P_2 \Big[ (\tau - t) \Leb^N(B_\lambda)\Big]^{1/2} C_\eps  \\
& +  \frac{p}{p-1}2 \big( L + \overline L \big) P_1 \eps
\left[ 1 + \log^+ \left(   \frac{ (L+\overline L)^{1/p}\| b \|_{L^p((t,\tau) \times B_\lambda)}}{(L+\overline L)P_1 \eps} 
\frac{\left[ (\tau-t) \Leb^N (B_r) \right]^{1-1/p}}{\delta} \right) \right].
\end{split}
\end{equation}
But according to the definition \eqref{e:functional} of $\Phi_\delta(\tau)$, given $\gamma>0$, we observe that
\begin{equation*}
\begin{split}
\Phi_\delta(\tau) \, \geq \, & \int_{B_r \cap \{ |X(\tau,x)-\overline X(\tau,x)| > \gamma\} \cap G_\lambda \cap \overline G_\lambda} 
\log \left( 1 +\frac{\gamma}{\delta} \right) \, dx \\
= \, & \log \left( 1+\frac{\gamma}{\delta} \right) \Leb^N \Big( B_r \cap \big\{ |X(\tau,\cdot)-\overline X(\tau,\cdot)| > \gamma \big\} \cap G_\lambda \cap \overline G_\lambda \Big) \,,
\end{split}
\end{equation*}
which implies
\begin{equation}\label{e:basiclevel}
\begin{split}
\Leb^N \Big( B_r \cap & \big\{ |X(\tau,\cdot)-\overline X(\tau,\cdot)| > \gamma \big\} \Big) \\
\leq \, & \Leb^N \Big( B_r \cap \big\{ |X(\tau,\cdot)-\overline X(\tau,\cdot)| > \gamma \big\} \cap G_\lambda \cap \overline G_\lambda \Big) + \Leb^N (B_r \setminus G_\lambda) + \Leb^N (B_r \setminus \overline G_\lambda) \\
\leq \, & \frac{\Phi_\delta(\tau)}{ \displaystyle \log \left( 1+\frac{\gamma}{\delta} \right)}
 + \Leb^N (B_r \setminus G_\lambda) + \Leb^N (B_r \setminus \overline G_\lambda) \,.
\end{split}
\end{equation}
Combining \eqref{e:basiclevel} and \eqref{e:integrata} we get
\begin{equation}
\begin{split}
\Leb^N \Big( B_r \cap & \big\{ |X(\tau,\cdot)-\overline X(\tau,\cdot)| > \gamma \big\} \Big) \\
\leq \, & \frac{\overline L}{\delta  \log \left( 1+\frac{\gamma}{\delta} \right)} 
\| b - \overline b \|_{L^1((t,\tau) \times B_\lambda)} + \frac{ L + \overline L }{\log \left( 1+\frac{\gamma}{\delta} \right)} P_2 \Big[ (\tau - t) \Leb^N(B_\lambda)\Big]^{1/2} C_\eps \\
& +  2\frac{p}{p-1} \frac{( L + \overline L ) P_1 \eps}{\log \left( 1+\frac{\gamma}{\delta} \right)}
\left[ 1 + \log^+ \left(   \frac{ (L+\overline L)^{1/p}\| b \|_{L^p((t,\tau) \times B_\lambda)}}{(L+\overline L)P_1 \eps} 
\frac{\left[ (\tau-t) \Leb^N (B_r) \right]^{1-1/p}}{\delta} \right) \right] \\ 
&  + \Leb^N (B_r \setminus G_\lambda) + \Leb^N (B_r \setminus \overline G_\lambda) \\
= \, & I + II + III + IV + V \,.
\end{split}
\end{equation}
We are now ready to conclude. Let use fix $\eta >0$. According to Lemma~\ref{l:superlevels},
we can choose $\lambda>0$ large enough to ensure that $IV \leq \eta/4$ and $V \leq \eta/4$.
We then consider $III$. We can find $\eps>0$ small enough in such a way that $ III \leq \eta/4$
for every $0<\delta \leq\gamma$ (notice that $III$ is uniformly bounded as $\delta \downarrow 0$).
Since at this point $\lambda$ and $\eps$ (and thus $C_\eps$) are fixed,
we choose $\delta>0$ small enough in such a way that $ II \leq \eta/4$. By setting
$$ C_{\gamma,r,\eta} =  \frac{\overline L}{\delta  \displaystyle \log \left( 1+\frac{\gamma}{\delta} \right)} \,, $$
where $\delta>0$ has been chosen according to the above discussion, the proof is completed.
\end{proof}

\section{Regular Lagrangian flow: existence, uniqueness, stability and further properties}\label{s:consequences}

In this section we show how the estimate in Proposition~\ref{l:main} implies the well-posedness
and further properties of the regular Lagrangian flow. We start by showing uniqueness and stability.

\begin{theorem}[Uniqueness]\label{t:unirlf}
Let $b$ be a vector field satisfying assumptions (R1), (R2) and (R3), and fix $t \in [0,T)$. Then, the regular Lagrangian flow associated to $b$ starting at time $t$, if it exists, is unique.
\end{theorem}
\begin{proof} It is a straightforward consequence of Proposition~\ref{l:main}.
Indeed, consider two regular Lagrangian flows $X$, $\overline X$ associated to $b$ and starting at time $t$, with compression constants $L$ and $\bar L$. Then we obtain the validity of \eqref{e:mainlemma} with $b = \bar b$. Namely, for every $\gamma > 0$ and every $r>0$, there holds
$$
	\Leb^N \Big( B_r \cap \big\{ |X(s,\cdot) - \overline X (s,\cdot) | > \gamma \big\} \Big) \leq \eta
	\qquad \mbox{for all }\eta > 0\mbox{ and }s\in[t,T].
$$
This implies that $X=\overline X$.
\end{proof}

\begin{theorem}[Stability]\label{t:stabrlf}
Let $\{ b_n \}$ be a sequence of vector fields satisfying assumption (R1), converging in $L^1_\loc([0,T] \times \R^N)$
to a vector field $b$ which satisfies assumptions (R1), (R2) and (R3).
Assume that there exist $X_n$ and $X$ regular Lagrangian flows starting at time $t$ associated to $b_n$
and to $b$ respectively, and denote by $L_n$ and $L$ the compression constants of the flows.
Suppose that:
\begin{itemize}
\item For some decomposition $b_n / (1+|x|) = \tilde b_{n,1} + \tilde b_{n,2}$ as in assumption (R1), we have that
$$
\| \tilde b_{n,1} \|_{L^1(L^1)} + \| \tilde b_{n,2} \|_{L^1(L^\infty)}
\qquad \text{ is equi-bounded in $n$;}
$$
\item The sequence $\{ L_n \}$ is equi-bounded.
\end{itemize}
Then the sequence $\{ X_n \}$ converges to $X$ locally in measure in $\R^N$, uniformly with respect to $s$ and $t$.
\end{theorem}
\begin{proof} We apply Proposition~\ref{l:main} with $\bar b=b_n$. 
According to our assumptions, the constants $\lambda > 0$ and $C_{\gamma,r,\eta}$ can be chosen independently of $n$.
Thus, we can choose $\bar n$ large enough in such a way that
$$
C_{\gamma,r,\eta}\| b - b_n \|_{L^1((0,T) \times B_\lambda)} \leq \eta
\qquad \text{ for all $n \geq \bar n$.}
$$
This means that, given any $r>0$ and any $\gamma>0$, for every $\eta>0$ we can find $\bar n$ for which
$$
\Leb^N \Big( B_r \cap \big\{ |X(s,\cdot) - X_n (s,\cdot) | > \gamma \big\} \Big)
\leq 2 \eta \qquad \text{ for all } n \geq \bar n\mbox{ and } s\in[t,T].
$$ 
This is precisely the desired thesis.
\end{proof}

The existence of the regular Lagrangian flow follows by an approximation procedure,
again with the help of Proposition~\ref{l:main} in order to to derive a compactness estimate.

\begin{lemma}[Compactness]\label{Lemma compact}
Let $\{ b_n \}$ be a sequence of vector fields satisfying assumption (R1), (R2) and (R3), converging in $L^1_\loc([0,T] \times \R^N)$
to a vector field $b$ which satisfies assumptions (R1), (R2) and (R3).
Assume that there exist $X_n$ regular Lagrangian flows starting at time $t$ associated to $b_n$,
and denote by $L_n$ the compression constants of the flows.
Suppose that:
\begin{itemize}
\item For some decomposition $b_n / (1+|x|) = \tilde b_{n,1} + \tilde b_{n,2}$ as in assumption (R1), we have that
$$
\| \tilde b_{n,1} \|_{L^1(L^1)} + \| \tilde b_{n,2} \|_{L^1(L^\infty)}
\qquad \text{ is equi-bounded in $n$;}
$$
\item The sequence $\{ L_n \}$ is equi-bounded;
\item For some $p>1$ the norms $\|b_n\|_{L^p((0,T)\times B_r)}$ are equi-bounded for any fixed $r>0$;
\item The norms of the singular integral operators associated to the vector fields $b_n$ (as well as their number $m$) are equi-bounded;
\item The functions $g_{jk}^n$ are uniformly in $n$ equi-integrable in $L^1 \big( (0,T)\times\R^N \big)$.
\end{itemize}
Then the sequence $\{ X_n \}$ converges as $n\rightarrow\infty$ to some $X$ locally in measure in $\R^N$, uniformly with respect to $s$ and $t$,
and $X$ is a regular Lagrangian flow starting at time $t$ associated to $b$.
\end{lemma}
\begin{proof} The application of Proposition~\ref{l:main} to the vector fields
$b_n$ and $b_m$ yields that for any $r>0$ and $\gamma>0$
$$ \Leb^N \big( B_r\cap\left\{|X_n(s,\cdot)-X_m(s,\cdot)|>\gamma\right\} \big)
	\rightarrow 0\qquad\mbox{as }m,n\rightarrow\infty,\mbox{ uniformly in }s,t. $$
Thus there exists $X\in C\big( [t,T]_s ; L^0_\loc(\R^N_x) \big)$ such that
$X_n\rightarrow X$ locally in measure in $\R^N$, uniformly in $s,t$.
The bound \eqref{eq:logpoint} being uniform in $n$ and $s,t$, we deduce that
$X\in {\mathcal B}\big( [t,T]_s ; \log L_\loc(\R^N) \big)$. It satisfies obviously (ii)
in Definition~\ref{d:rlf}. For (iii), it is enough to prove \eqref{e:comprfunct} for $\varphi\in C_c(\R^N)$,
$\varphi\geq 0$, and it follows from Fatou's lemma, with $L=\liminf L_n$.
Finally, for proving (i) in Definition~\ref{d:rlf}, it is enough to get \eqref{e:rlfeq}
for $\beta\in C^1_c(\R^N)$, because a general $\beta$ can be approximated by
$\beta_\epsilon(z)=\beta(z)\chi\bigl(\epsilon\log(2+|z|^2)\bigr)$, where $\chi\in C^\infty_c([0,\infty))$,
$\chi\geq 0$, $\chi(y)=1$ for $y\leq 1$ (use Lebesgue's theorem).
Now, for $\beta\in C^1_c(\R^N)$, we can pass to the limit from the equation \eqref{e:rlfeq}
written for $X_n$, because then $X_n(s,t,x)$ lies in a fixed ball $B_r$ (the support of $\beta$), and we have the uniform bound
$\|b_n\|_{L^p((0,T)\times B_r)}$, which implies local equi-integrability. Thus in this context it is enough to prove the local convergence
in measure of $\beta(X_n)$ and $\beta'(X_n)b_n(s,X_n)$ to $\beta(X)$ and $\beta'(X)b(s,X)$ respectively,
which can be obtained with standard analysis (see for example the proof of Theorem~\ref{t:existrlf} below,
with the use of Lusin's theorem).
Therefore $X$ is a regular Lagrangian flows starting at time $t$ associated to $b$.
\end{proof}

In order to find a sequence of approximations $b_n$ to $b$ with uniformly
bounded compression constants, a convenient method is to make a further assumption
on the divergence of the vector field, as we describe in the following
theorem (but see also Remark~\ref{r:abstrapprox}).

\begin{theorem}[Existence]\label{t:existrlf}
Let $b$ be a vector field satisfying assumptions (R1), (R2) and (R3), and assume that 
\begin{equation}\label{e:ipodiv}
\div b \geq \alpha(t)\quad\mbox{in }(0,T)\times\R^N, \qquad \text{ with $\alpha \in L^1(0,T)$.}
\end{equation} 
Then, for all $t \in [0,T)$ there exists a regular Lagrangian flow associated to $b$ starting at time $t$.
Moreover,  the flow $X$ satisfies
\begin{equation}\label{e:contflow} 
X\in C \big( D_T ; L^0_\loc(\R^N_x) \big)\cap {\mathcal B} \big( D_T ; \log L_\loc(\R^N_x) \big) \,,
\end{equation}
where $D_T=\left\{(s,t);0\leq t\leq s\leq T\right\}$,
and for every $0\leq t\leq\tau\leq s\leq T$ there holds
\begin{equation}\label{e:semigroupplus}
X \big( s,\tau, X(\tau,t,x) \big) = X (s,t,x)
\qquad \text{ for $\Leb^N$-a.e.~$x \in \R^N$.}
\end{equation}
\end{theorem}
\begin{proof} We fix a positive radial convolution kernel $\zeta$ in $\R^N$, with $\spt \zeta \subset B_1$.
By defining $b_n = b \mathop{*}\limits_x \zeta_n$,
it is immediate to check that  there exist decompositions $b_n / (1+|x|) = \tilde b_{n,1} + \tilde b_{n,2}$ as in \eqref{e:grass2} for which 
$$
\| \tilde b_{n,1} \|_{L^1(L^1)} + \| \tilde b_{n,2} \|_{L^1(L^\infty)}
\qquad \text{ is equi-bounded in $n$,}
$$
and that the sequence $\{ b_n \}$ is equi-bounded in $L^p_\loc([0,T] \times \R^N)$.
Moreover, we have
$$
\partial_j b_n = \sum_{k=1}^m S_{jk} \big( g_{jk} \big)_n \qquad \text{in $\D' \big( (0,T) \times \R^N \big)$,}
$$
where we have set
$$
\big( g_{jk} \big)_n = g_{jk} \mathop{*}\limits_x \zeta_n \,.
$$
Thus, for all $j=1,\ldots,N$ and $k=1,\ldots,m$, the family $\{ (g_{jk})_n \}_{n \in \N}$ is bounded in 
$L^1 \big( (0,T)\times\R^N \big)$ and equi-integrable in $(0,T) \times \R^N$
(indeed it converges strongly in $L^1 \big( (0,T)\times\R^N \big)$ to $g_{jk}$).

Let us consider, for every $n$, the regular Lagrangian flow $X_n$ associated to $b_n$ starting at time $t$, which indeed is a classical flow, being each $b_n$ smooth with respect to $x$.
By \eqref{e:ipodiv} we have $\div b_n = (\div b) * \zeta_n\geq\alpha(t)$,
thus we can take for compression constant of $X_n$ the value
$L_n=\exp \big( \|\alpha\|_{L^1(0,T)} \big)$, independently on $n$.
Applying Lemma~\ref{Lemma compact} yields the existence result, with
$L = \exp \big( \|\alpha\|_{L^1(0,T)} \big)$.

The continuity statement in \eqref{e:contflow} follows from the uniform convergence result of Lemma~\ref{Lemma compact},
and the boundedness in $\log L_\loc(\R^N)$ comes from the uniform in time estimate \eqref{eq:logpoint}.

It remains to prove \eqref{e:semigroupplus}. Note that the composition in this formula is well-defined almost everywhere
in $x$ because of \eqref{eq:absmeas} (take $B$ of zero measure).
We still use the approximation of $b$ by a sequence $b_n$ of smooth vector fields, in such a way that $X_n(s,t,\cdot)$
converges towards $X(s,t,\cdot)$ locally in measure in $\R^N$,
uniformly in $(s,t) \in D_T$. Clearly \eqref{e:semigroupplus} holds for the approximating flows $X_n$. In order to pass to the limit we show that 
\begin{equation}\label{e:claimconvmeas}
X_n \big( s,\tau, X_n(\tau,t,\cdot) \big) \to X \big( s,\tau, X(\tau,t,\cdot) \big)
\qquad \text{locally in measure in $\R^N$.}
\end{equation}
Let us fix $\eta >0$, $\gamma>0$ and $r>0$.
According to Lemma ~\ref{l:superlevels} we can find $\lambda\geq r$
depending on $\eta$ and $r$ in such a way that for all $s$ and $t$
$$
\Leb^N (B_r \setminus G_\lambda^{s,t}) \leq \eta \qquad
\text{ and } \qquad \Leb^N (B_r \setminus G^{s,t,n}_\lambda) \leq \eta \,,
$$ 
where we denote by $G_\lambda^{s,t}$ and $G^{s,t,n}_\lambda$ the sublevels at fixed time of $X$ and $X_n$ respectively,
i.e. $G_\lambda^{s,t}=\{x\in\R^N \, : \, |X(s,t,x)|\leq\lambda\}$
(the inequality \eqref{eq:estsuperlevels} is valid for these sublevels, just use \eqref{eq:logpoint}
instead of \eqref{e:loggrowth}).

Since $X_n(\tau,t,\cdot) \to X(\tau, t,\cdot)$ locally in measure and uniformly in $\tau,t$, we can find $n_1$
such that for all $n\geq n_1$, and all $\tau,t$,
$$ \Leb^N \big(B_\lambda\backslash S_1^{n,\tau,t}\big)\leq\eta \,,
	\qquad\mbox{with }S_1^{n,\tau,t}=\big\{x\in B_\lambda \, : \, |X_n(\tau,t,x)-X(\tau,t,x)|\leq\gamma/4\big\} \,.
$$
Then we have
\begin{equation}\label{e:estconv2}
\begin{split}
\Leb^N \Big( B_r \cap \big\{ \big| X_n\big(s,\tau,X_n(\tau,t,\cdot)\big) - X\big(s,\tau,X(\tau,t,\cdot)\big) \big| > \gamma \big \}\Big) \\
\leq \Leb^N \Big( B_r\cap G^{s,t,n}_\lambda \cap \big\{ \big| X_n\big(s,\tau,X_n(\tau,t,\cdot)\big) -  X\big(s,\tau,X_n(\tau,t,\cdot)\big) \big| > \gamma/4 \big \}\Big) \mkern 50mu\\
+ \Leb^N \Big( B_r\cap G^{s,t}_\lambda\cap G^{s,t,n}_\lambda\cap \Big\{ \big|X\big(s,\tau,X_n(\tau,t,\cdot)\big)  - X\big(s,\tau,X(\tau,t,\cdot)\big) \big| > \frac{3}{4}\gamma \Big\}\Big)\\
+\Leb^N (B_r \setminus G^{s,t}_\lambda)+\Leb^N (B_r \setminus G^{s,t,n}_\lambda).
\end{split}
\end{equation}
The first term of the right-hand side is bounded by $L_n\Leb^N \big(B_\lambda\backslash S_1^{n,s,\tau}\big)\leq L_n\eta$ (for $n\geq n_1$),
and the two last terms are bounded by $\eta$.
Now, according to Lusin's theorem we can find $\hat X$ such that $\hat X (s,\tau,\cdot) \in C(\overline{B_\lambda})$ and
\begin{equation}
	\Leb^N \Big( \overline{B_\lambda} \cap \big\{ X (s,\tau ,\cdot) \not = \hat X (s,\tau ,\cdot) \big\} \Big) \leq \eta.
	\label{eq:lusin}
\end{equation}
Then, since $\hat X(s,\tau ,\cdot)$ is uniformly continuous, there exists $\alpha>0$ such that
\begin{equation}
	|y-x|\leq\alpha\quad\mbox{implies that}\quad |\hat X(s,\tau ,y)-\hat X(s,\tau ,x)|\leq\frac{3}{4}\gamma.
	\label{eq:unifcontX}
\end{equation}
Then there exists $n_2$ such that for all $n\geq n_2$
$$ \Leb^N \big(B_r\backslash S_2^{n,\tau,t}\big)\leq\eta\,,
	\qquad\mbox{with }S_2^{n,\tau,t}=\big\{x\in B_r \, : \, |X_n(\tau,t,x)-X(\tau,t,x)|\leq\alpha\big\} \,.
$$
We can estimate the second term in the right-hand side of \eqref{e:estconv2} as
\begin{equation*}
\begin{split}
\Leb^N \Big( B_r\cap G^{s,t}_\lambda\cap G^{s,t,n}_\lambda\cap \Big\{ \big|X\big(s,\tau,X_n(\tau,t,\cdot)\big)  - X\big(s,\tau,X(\tau,t,\cdot)\big) \big| > \frac{3}{4}\gamma \Big\}\Big)\\
\leq\Leb^N \Big( B_r\cap G^{s,t}_\lambda\cap G^{s,t,n}_\lambda\cap \Big\{ \big|\hat X\big(s,\tau,X_n(\tau,t,\cdot)\big)  - \hat X\big(s,\tau,X(\tau,t,\cdot)\big) \big| > \frac{3}{4}\gamma \Big\}\Big)\\
+\Leb^N \Big( B_r\cap G^{s,t,n}_\lambda\cap \big\{ \hat X\big(s,\tau,X_n(\tau,t,\cdot)\big)  \not =  X\big(s,\tau,X_n(\tau,t,\cdot)\big) \big \}\Big)\\
+\Leb^N \Big( B_r\cap G^{s,t}_\lambda\cap \big\{ \hat X\big(s,\tau,X(\tau,t,\cdot)\big)  \not =  X\big(s,\tau,X(\tau,t,\cdot)\big) \big \}\Big).
\end{split}
\end{equation*}
Taking into account \eqref{eq:lusin}, the two last terms are bounded respectively by $L_n\eta$
and $L\eta$. Because of \eqref{eq:unifcontX}, the first term on the right-hand side is bounded by $\Leb^N \big(B_r\backslash S_2^{n,\tau,t}\big)$.
We conclude that for $n\geq \max\{n_1,n_2\}$
$$ \Leb^N \Big( B_r \cap \big\{ \big| X_n\big(s,\tau,X_n(\tau,t,\cdot)\big) - X\big(s,\tau,X(\tau,t,\cdot)\big) \big| > \gamma \big \}\Big)
	\leq C\eta \,, $$
which is exactly the desired convergence in measure in \eqref{e:claimconvmeas}.
Together with the local convergence in measure in $\R^N$ of $X_n(s,t,\cdot)$ towards $X(s,t,\cdot)$, this proves \eqref{e:semigroupplus}.
\end{proof}

\begin{remark}\label{r:abstrapprox} In the previous theorem we assume the condition \eqref{e:ipodiv} in order to be sure to have a smooth approximating sequence with equi-bounded compression constants.
An assumption on the divergence is the easiest (and the most explicit) way to get such equi-bound, due to the fact that the regularization by convolution preserves the $L^\infty$ bounds on the divergence.
However, according to Lemma~\ref{Lemma compact}, a sharp assumption for the existence of a regular Lagrangian flow would be the existence of a smooth approximating sequence with equi-bounds on the growth assumptions in (R1) and (R3),
that satisfies (R2) with fixed singular integral operators $S_{jk}$ (or at least with singular integral operators satisfying uniform bounds)
and with equi-integrable functions $g_{jk}$, and for which the compression constants are equi-bounded.
Observe that (by a diagonal argument) the class of vector fields that enjoy this approximation property is closed with respect to such convergence with bounds.\end{remark}

Summing up all the previous results, we have existence, uniqueness and stability of the (forward) regular Lagrangian flow
starting at time $t \in [0,T)$, associated to a vector field $b$ satisfying assumptions (R1), (R2) and (R3) and for which \eqref{e:ipodiv} holds.
In the following corollary we deal with the case when two-sided bounds on the divergence are assumed,
and thus we can define forward and backward flows, that also satisfy the usual group property.

\begin{corol}[Forward-backward flow]\label{c:finalwell}
Let $b$ be a vector field satisfying assumptions (R1), (R2) and (R3) and assume that \eqref{e:ipodiv} is replaced by the stronger condition
\begin{equation}\label{e:ipodiv3}
\div b \in L^1 \big( (0,T) ; L^\infty( \R^N ) \big) \,.
\end{equation}
Then, for every $t \in [0,T]$, there exists a unique forward and backward (i.e., satisfying the conditions of Definition~\ref{d:rlf} for $s\leq t$)
regular Lagrangian flow associated to $b$ starting at time $t$.
Such flow $X$, as a function of both $s$ and $t$, satisfies
\begin{equation}\label{e:fbflow}
X \in C\big( [0,T]_s \times [0,T]_t ; L^0_\loc(\R^N_x) \big) \cap
{\mathcal B}\big( [0,T]_s \times [0,T]_t ; \log L_\loc(\R^N_x) \big) \,.
\end{equation}
Moreover, for every $s$, $t$ and $\tau$ in $[0,T]$ there holds
\begin{equation}\label{e:semigroup}
X \big( s,\tau, X(\tau,t,x) \big) = X (s,t,x)
\qquad \text{ for $\Leb^N$-a.e.~$x \in \R^N$,}
\end{equation}
and in particular for all $s$, $t \in [0,T]$
\begin{equation}\label{e:semiinverse}
X \big( t,s,X(s,t,x) \big) = x
\qquad \text{ for $\Leb^N$-a.e.~$x \in \R^N$.}
\end{equation}
\end{corol}

\begin{proof}The two-sided bound \eqref{e:ipodiv3} enables to apply Theorem~\ref{t:existrlf}
to $b$ and to $\overline b(t,x)=-b(T-t,x)$.
This gives the existence and uniqueness of the forward and backward flows, with a control
on the compression constants $L\leq\exp\|\div b\|_{L^1 ( (0,T) ; L^\infty( \R^N ) )}$.
Condition \eqref{e:fbflow} follows.

For proving \eqref{e:semigroup}, we have to consider an approximation of $b$ by a sequence $b_n$ of smooth vector fields, in such a way that $X_n(s,t,\cdot)$
converges towards $X(s,t,\cdot)$ locally in measure in $\R^N$,
uniformly in $s$, $t \in [0,T]$, and with uniform bounds on the compression.
This can be done by the same approximation as in the proof of Theorem~\ref{t:existrlf}.
In this way we have a simultaneous approximation of the forward and backward flows.
Then it is straightforward to check that the proof of Theorem~\ref{t:existrlf} works as well
for all values of $s,\tau,t$ in $[0,T]$.
\end{proof}

We finally would like to define the Jacobian of the flow, which is by definition a bounded map
$JX(s,t,\cdot)\in L^\infty(\R^N)$ satisfying for all $s$, $t \in [0,T]$ the equality
\begin{equation}\label{e:jacsmooth}
\int_{\R^N} f(y) \, dy = \int_{\R^N} f \big( X(s,t,x) \big) JX (s,t,x) \, dx
\qquad \text{ for every $f \in L^1(\R^N)$.}
\end{equation}

\begin{propos}[Jacobian of the flow]\label{p:jacob}
Consider a vector field $b$ satisfying assumptions (R1), (R2) and (R3) and such that
\begin{equation}\label{e:ipodiv4}
\div b \in L^1 \big( (0,T) ; L^\infty(\R^N) \big) \,.
\end{equation}
Then there exists a unique Jacobian $JX$ satisfying \eqref{e:jacsmooth}.
Moreover,
\begin{equation}
	JX \in C \big( [0,T]_s \times [0,T]_t ; L^\infty(\R^N)-w* \big)
	\cap C \big( [0,T]_s \times [0,T]_t ; L^1_\loc(\R^N) \big),
	\label{eq:regJ}
\end{equation}
and $JX\geq 0$.
\end{propos}
\begin{proof} {\sc Step 1. Uniqueness.} Assume there are two Jacobians
$\mathcal J_{s,t}$ and $\mathcal J'_{s,t}$. Then
\begin{equation}\label{e:JJ2}
\int_{\R^N} f \big( X(s,t,x) \big) \mathcal J_{s,t} (x) \, dx = 
\int_{\R^N} f \big( X(s,t,x) \big) \mathcal J'_{s,t} (x) \, dx
\qquad \text{ for every $f \in L^1(\R^N)$.}
\end{equation}
Given $g \in L^1(\R^N)$, let us take $f(y) = g \big( X(t,s,y) \big)$.
Since the forward and backward flows have bounded compression, we have $f \in L^1(\R^N)$.
Thus we can substitute in \eqref{e:JJ2}, and with \eqref{e:semiinverse} we obtain
$$
\int_{\R^N} g(x) \, \mathcal J_{s,t} (x) \, dx = 
\int_{\R^N} g(x) \, \mathcal J'_{s,t} (x) \, dx
\qquad \text{ for every $g \in L^1(\R^N)$.}
$$
This implies that $\mathcal J_{s,t}=\mathcal J'_{s,t}$ a.e. in $\R^N$.

\medskip

{\sc Step 2. Existence.}
We consider as in Theorem~\ref{t:existrlf} and Corollary~\ref{c:finalwell} a smooth approximation
$b_n$ of $b$ with uniform bounds.
For every $n$ we have a regular Lagrangian flow $X_n$ associated to $b_n$, and its Jacobian $JX_n$.
The a priori bound $\|JX_n(s,t,\cdot)\|_{L^\infty}\leq \exp\|\div b_n\|_{L^1(L^\infty)}\leq \Lambda$
holds and since $X_n$ is a classical flow we have
\begin{equation}\label{e:JJ}
\int_{\R^N} f(y) \, dy = \int_{\R^N} f \big( X_n(s,t,x) \big) JX_n(s,t,x) \, dx
\qquad \text{ for every $f \in L^1(\R^N)$.}
\end{equation}
Next, we claim that for all $f\in L^1(\R^N)$,
\begin{equation}
	f \big(X_n(s,t,\cdot) \big) \mbox{ is equi-integrable in }L^1(\R^N)\mbox{ with respect to }n,s,t.
	\label{eq:equifX}
\end{equation}
By density, it is enough to prove this for $f\in C_c(\R^N)$. In this case, we just have to prove that
$$
	\sup_{s,t,n}\int_{|x|>r}\left|f(X_n(s,t,x))\right|\,dx\rightarrow 0 \qquad\mbox{as }r\rightarrow \infty.
$$
Since $f$ has support in a ball $B_\lambda$, it amounts to prove that for $\lambda>0$,
$$ \sup_{s,t,n}\Leb^N\left(\left\{x\in\R^N \, : \, |x|>r\mbox{ and }|X_n(s,t,x)|\leq \lambda\right\}\right) \rightarrow 0
	\qquad\mbox{as }r\rightarrow \infty ,
$$
which holds true according to Remark~\ref{r:Xsmallinfinity}.
Thus \eqref{eq:equifX} is proved. Then we claim that for all $f\in L^1(\R^N)$
\begin{equation}
	f \big( X_n(s,t,\cdot) \big) \rightarrow f \big( X(s,t,\cdot) \big) 
	\quad \mbox{ in }L^1(\R^N)\mbox{ uniformly in }s,t,
	\label{eq:convunif}
\end{equation}
as $n\rightarrow\infty$. Again by density, it is enough to prove it for $f\in C_c(\R^N)$,
and this follows from \eqref{eq:equifX} and the fact that $f \big( X_n(s,t,\cdot) \big)\rightarrow f \big( X(s,t,\cdot) \big)$
locally in measure and uniformly in $s,t$. Thus \eqref{eq:convunif} holds.
Next, taking as above $f(y)=g \big(X(t,s,y) \big)$ yields that \eqref{e:jacsmooth} is equivalent to
finding $JX(s,t,\cdot)\in L^\infty(\R^N)$ such that
\begin{equation}
	\int_{\R^N} g\bigl(X(t,s,y)\bigr) \, dy = \int_{\R^N} g(x) JX (s,t,x) \, dx
\qquad \text{ for every $g \in L^1(\R^N)$,}
	\label{eq:Jg}
\end{equation}
and similarly we can transform \eqref{e:JJ} into
\begin{equation}
	\int_{\R^N} g\bigl(X_n(t,s,y)\bigr) \, dy = \int_{\R^N} g(x) JX_n (s,t,x) \, dx
\qquad \text{ for every $g \in L^1(\R^N)$.}
	\label{eq:Jgn}
\end{equation}
But because of \eqref{eq:convunif}, we have
$g\big(X_n(t,s,\cdot)\big)\rightarrow g\big(X(t,s,\cdot)\big)$ in $L^1(\R^N)$ uniformly in $s,t$,
thus $g\big(X(t,s,\cdot)\big)\in C\big([0,T]_s\times[0,T]_t;L^1(\R^N)\big)$ and we deduce that
the right-hand side of \eqref{eq:Jgn} converges uniformly in $s,t$ to some continuous function,
the left-hand side of \eqref{eq:Jg}. This being true for all $g\in L^1(\R^N)$,
for all $s,t$ we conclude that $JX_n(s,t,\cdot)$ converges in $L^\infty-w*$ to some
function $JX(s,t,\cdot)\in L^\infty$ that satisfies \eqref{eq:Jg}. Finally, the continuity
of \eqref{eq:Jg} with respect to $s,t$ yields that
$JX\in C\big([0,T]_s\times[0,T]_t;L^\infty(\R^N)-w*\big)$. The nonnegativity of $JX_n$
implies also that $JX\geq 0$.

\medskip

{\sc Step 3. Strong continuity.} It is only now that we really use the assumption
\eqref{e:ipodiv4} on $\div b$, and not only the bounded compression.
The Jacobians associated to the smooth flows $X_n$ satisfy the ordinary differential equation
$$ \partial_s JX_n(s,t,x) = (\div b_n) (s,X_n(s,t,x)) \, JX_n(s,t,x) \,, $$
thus we can write
$$ JX_n(s,t,x)=\exp\int_t^s(\div b_n)\bigl(\tau,X_n(\tau,t,x)\bigr)\,d\tau \,. $$
We have
$$ \div b_n\longrightarrow \div b\quad\mbox{in }L^1_\loc \big( [0,T]\times\R^N \big) \,, $$
hence also locally in measure in $[0,T]\times\R^N$, and $X_n(\tau,t,x)\rightarrow X(\tau,t,x)$ locally in measure,
uniformly in $\tau,t$. Since $\div b_n$ is bounded in $L^1 \big( (0,T);L^\infty(\R^N) \big)$, we deduce that
$$ (\div b_n) \big( \tau,X_n(\tau,t,x)\big)\longrightarrow(\div b)\big(\tau,X(\tau,t,x)\big)\quad
	\mbox{in }L^1_\loc \big( [0,T]\times\R^N \big),\mbox{ uniformly in }t.
$$
Therefore
$$ \int_t^s(\div b_n)\bigl(\tau,X_n(\tau,t,x)\bigr)\,d\tau
	\longrightarrow \int_t^s(\div b)\bigl(\tau,X(\tau,t,x)\bigr)\,d\tau\quad
	\mbox{in }L^1_\loc(\R^N),\mbox{ uniformly in }s,t, $$
and
\begin{equation}
	JX_n(s,t,x)\rightarrow\exp\int_t^s(\div b)\bigl(\tau,X(\tau,t,x)\bigr)\,d\tau\quad
	\mbox{in }L^1_\loc(\R^N),\mbox{ uniformly in }s,t.
	\label{eq:convJ_n}
\end{equation}
We conclude that $JX$ is equal to the right-hand side of \eqref{eq:convJ_n}, that
$JX_n(s,t,x)\rightarrow JX(s,t,x)$ in $L^1_\loc(\R^N)$, uniformly in $s,t$, and that
$JX\in C \big( [0,T]_s \times [0,T]_t ; L^1_\loc(\R^N) \big)$.
\end{proof}

\section{Lagrangian solutions to the transport and continuity equations}\label{s:lagrangian}

In this final section we introduce the concept of Lagrangian solutions to the transport and continuity equations.
They are defined as superposition of the initial data with the regular Lagrangian flow of the ordinary differential equation.
The well-posedness results stated in the previous section yield
that such Lagrangian solutions are well-defined and stable.
These solutions are in particular solutions in the renormalized distributional sense.

Note that under our assumptions on the coefficient $b$
we do not know if a solution in the renormalized or distributional
sense is unique. It could in principle happen that there exist several distributional solutions,
only one of them being associated to the flow.
Introducing the notion of Lagrangian solution, we identify
(among the many possible distributional solutions) a unique stable semigroup of solutions. This will be relevant
for the applications in the forthcoming paper \cite{BC2}.

\medskip

We first consider the backward Cauchy problem for the transport equation with prescribed final data, that reads
\begin{equation}\label{e:transT}
\begin{cases}
\partial_t u + b \cdot \nabla u = 0\quad\mbox{in }(0,T)\times\R^N, \\
u(T,\cdot) = u^T \,.
\end{cases}
\end{equation}
\begin{definition}[Lagrangian solution to the transport equation]\label{d:transport} Assume that $b$ satisfies (R1), (R2), (R3), and
$\div b \geq \alpha(t)$ with $\alpha \in L^1(0,T)$.
If $u^T \in L^0(\R^N)$, we define the Lagrangian solution to \eqref{e:transT} by
\begin{equation}
	u(t,x) = u^T \big( X(T,t,x) \big) \,.
	\label{eq:deftransp}
\end{equation}
\end{definition}
According to Theorems~\ref{t:unirlf} and~\ref{t:existrlf}, there exists
a unique forward regular Lagrangian flow $X$ associated to $b$. The bounded compression
condition ensures that the function \eqref{eq:deftransp} will be modified only on a set
of measure zero if we change of representative of $u^T$, justifying the definition \eqref{eq:deftransp}.
This definition of course gives the classical solution in case of smooth data.
\begin{propos}\label{p:transport} Assume that $b$ satisfies (R1), (R2), (R3), and
$\div b \geq \alpha(t)$ with $\alpha \in L^1(0,T)$. Then the Lagrangian solution \eqref{eq:deftransp}
satisfies
\begin{itemize}
\item[(i)] For all $u^T \in L^0(\R^N)$, we have $u\in C\big([0,T];L^0_\loc(\R^N)\big)$;
\item[(ii)] For all $u^T\in L^q(\R^N)$ for some $1\leq q<\infty$, we have
$u\in C \big([0,T];L^q(\R^N)\big)$. For all $u^T\in L^\infty(\R^N)$, we have
$u\in  C\big([0,T];L^\infty(\R^N)-w*\big)\cap C\big([0,T];L^1_\loc(\R^N)\big)$;
\item[(iii)] If $\div b \in L^1_\loc \big( (0,T)\times\R^N \big)$, then for all $u^T \in L^0(\R^N)$
and $\beta\in C_b(\R)$,
$$
	\partial_t\bigl(\beta(u)\bigr)+b\cdot\nabla\bigl(\beta(u)\bigr)=0\quad
	\mbox{in }(0,T)\times\R^N,
$$
where we define $b\cdot \nabla v\equiv\div (bv) - v \, \div b$.
\end{itemize}
\end{propos}
\begin{proof} For (i), fix a ball $B_r$ and $\gamma>0$, $\eta>0$.
According to Lemma~\ref{l:superlevels}, we can find $\lambda>0$ such that $\Leb^N(\{x\in B_r \,:\, |X(T,t,x)|>\lambda\})\leq\eta$
for all $t\in[0,T]$. Then by Lusin's theorem, there exists $\hat u^T\in C(\overline B_\lambda)$ such that
$\Leb^N(\{y\in\overline B_\lambda \,:\, u^T(y)\not=\hat u^T(y)\})\leq\eta$. Since $\hat u^T$ is
uniformly continuous, there exists $\alpha>0$ such that $|\hat u^T(z)-\hat u^T(y)|\leq\gamma$
for all $y,z\in \overline B_\lambda$ such that $|z-y|\leq\alpha$. Then, given $t_0\in [0,T]$, there exists $\delta>0$
such that for $|t-t_0|\leq\delta$ we have $\Leb^N(\{x\in B_r \,:\, |X(T,t,x)-X(T,t_0,x)|>\alpha\})\leq\eta$.
Finally, for $|t-t_0|\leq\delta$ we have
\begin{equation*}\begin{array}{l}
	\dsp\hphantom{\leq} \Leb^N\left(\{x\in B_r \,:\, |u(t,x)-u(t_0,x)|>\gamma\}\right)\\
	\leq\Leb^N\left(\{x\in B_r \,:\, |X(T,t,x)|\leq\lambda,|X(T,t_0,x)|\leq\lambda,|u^T(X(T,t,x))-u^T(X(T,t_0,x))|>\gamma\}\right)+2\eta\\
	\leq\Leb^N\left(\{x\in B_r \,:\, |X(T,t,x)|\leq\lambda,|X(T,t_0,x)|\leq\lambda,|\hat u^T(X(T,t,x))-\hat u^T(X(T,t_0,x))|>\gamma\}\right)\\
	\dsp\mkern 600mu+2\eta+2L\eta\\
	\leq\Leb^N\left(\{x\in B_r \,:\, |X(T,t,x)-X(T,t_0,x)|>\alpha\}\right)+2\eta+2L\eta\\
	\leq 3\eta+2L\eta,
	\label{eq:estconttransp}
	\end{array}
\end{equation*}
which proves (i).

For (ii) and $q<\infty$, taking into account (i) we just have to prove that $(|u(t,\cdot)|^q)_{0\leq t\leq T}$ is equi-integrable.
This holds true by using Lemma~\ref{l:equiint}, the bounded compression property and Remark~\ref{r:Xsmallinfinity}.
The case $q=\infty$ is obvious.

For (iii), the statement involves only the function $\beta(u^T)$ (and not $\beta$ and $u^T$ separately).
Therefore we have to prove that for all $u^T\in L^\infty(\R^N)$,
\begin{equation}
	\partial_t \Bigl(u^T\big(X(T,t,x)\big)\Bigr)+b\cdot\nabla\Bigl(u^T\big(X(T,t,x)\big)\Bigr)=0\quad
	\mbox{in }(0,T)\times\R^N.
	\label{eq:transpbounded}
\end{equation}
By approximation it is enough to prove \eqref{eq:transpbounded} for $u^T\in C_c^\infty(\R^N)$.
This is obtained obviously by approximation of $b$ by a smooth sequence $b_n$ such that
$\div b_n\rightarrow\div b$ in $L^1_\loc\big((0,T)\times\R^N\big)$.
\end{proof}
\begin{propos}[Stability of the Lagrangian transport]\label{p:stabtransport} Let $b_n$, $b$ be vector fields satisfying assumptions
(R1), (R2), (R3), and $\div b_n\geq\alpha_n(t)$, $\div b\geq\alpha(t)$ for some $\alpha_n,\ \alpha\in L^1(0,T)$.
Assume that $b_n\rightarrow b$ in $L^1_\loc \big([0,T] \times \R^N\big)$, and that
\begin{itemize}
\item For some decomposition $b_n / (1+|x|) = \tilde b_{n,1} + \tilde b_{n,2}$ as in assumption (R1), we have that
$$
\| \tilde b_{n,1} \|_{L^1(L^1)} + \| \tilde b_{n,2} \|_{L^1(L^\infty)}
\qquad \text{ is equi-bounded in $n$;}
$$
\item The sequence $\|\alpha_n\|_{L^1(0,T)}$ is equi-bounded.
\end{itemize}
We consider the Lagrangian solutions $u_n(t,x)$ to the transport
equation with coefficient $b_n$ and final data $u_n^T\in L^0(\R^N)$, as well
as $u(t,x)$ associated to $b$ and $u^T\in L^0(\R^N)$. Then we have the following properties:
\begin{itemize}
\item[(i)] If $u_n^T\rightarrow u^T$ in $L^0_\loc(\R^N)$,
then $u_n(t,x)\rightarrow u(t,x)$ in $C\big([0,T];L^0_\loc(\R^N)\big)$;
\item[(ii)] If $u_n^T\rightarrow u^T$ in $L^q(\R^N)$ for some $1\leq q<\infty$,
then $u_n(t,x)\rightarrow u(t,x)$ in $C\big([0,T];L^q(\R^N)\big)$.
If $u_n^T\rightarrow u^T$ in $\big(L^\infty(\R^N)-w*\big)\cap L^1_\loc(\R^N)$, then
$u_n(t,x)\rightarrow u(t,x)$ in $C\big([0,T];L^\infty(\R^N)-w*\big)\cap C\big([0,T];L^1_\loc(\R^N)\big)$.
\end{itemize}
\end{propos}
\begin{proof}Denote by $X_n$ and $X$ the respective regular Lagrangian flows of $b_n$ and $b$.
Then according to Theorem~\ref{t:existrlf} we can take for their compression constants
$L_n=\exp\|\alpha_n\|_{L^1(0,T)}$ and $L=\exp\|\alpha\|_{L^1(0,T)}$, which are equi-bounded.
Thus we can apply Theorem~\ref{t:stabrlf}, and $X_n$ converges to $X$ locally in measure
in $\R^N$, uniformly in $s,t$.

For (i), fix a ball $B_r$ and $\gamma>0$, $\eta>0$.
According to Lemma~\ref{l:superlevels}, we can find $\lambda>0$ such that $\Leb^N(\{x\in B_r \,:\, |X_n(T,t,x)|>\lambda\})\leq\eta$
for all $t\in[0,T]$ and all $n$. Then, there exists some $n_1$ such that for all $n\geq n_1$,
$\Leb^N(\{y\in\overline B_\lambda \,:\, |u_n^T(y)-u^T(y)|>\gamma/2\})\leq\eta$.
By Lusin's theorem, there exists $\hat u^T\in C(\overline B_\lambda)$ such that
$\Leb^N(\{y\in\overline B_\lambda \,:\, u^T(y)\not=\hat u^T(y)\})\leq\eta$. Since $\hat u^T$ is
uniformly continuous, there exists $\alpha>0$ such that $|\hat u^T(z)-\hat u^T(y)|\leq\gamma/2$
for all $y,z\in \overline B_\lambda$ such that $|z-y|\leq\alpha$. Then, there exists $n_2$
such that for $n\geq n_2$ we have $\Leb^N(\{x\in B_r \,:\, |X_n(T,t,x)-X(T,t,x)|>\alpha\})\leq\eta$
for all $t\in[0,T]$.
Finally, for $n\geq\max(n_1,n_2)$ and $t\in[0,T]$ we have
\begin{equation*}\begin{array}{l}
	\dsp\hphantom{\leq} \Leb^N\left(\{x\in B_r \,:\, |u_n(t,x)-u(t,x)|>\gamma\}\right)\\
	\leq\Leb^N\left(\{x\in B_r \,:\, |X_n(T,t,x)|\leq\lambda,|X(T,t,x)|\leq\lambda,|u_n^T(X_n(T,t,x))-u^T(X(T,t,x))|>\gamma\}\right)+2\eta\\
	\leq\Leb^N\left(\{x\in B_r \,:\, |X_n(T,t,x)|\leq\lambda,| u_n^T(X_n(T,t,x))-u^T(X_n(T,t,x))|>\gamma/2\}\right)\\
	\dsp+\Leb^N\left(\{x\in B_r \,:\, |X_n(T,t,x)|\leq\lambda,|X(T,t,x)|\leq\lambda,| u^T(X_n(T,t,x))-u^T(X(T,t,x))|>\gamma/2\}\right)+2\eta\\
	\leq\Leb^N\left(\{x\in B_r \,:\, |X_n(T,t,x)|\leq\lambda,|X(T,t,x)|\leq\lambda,| \hat u^T(X_n(T,t,x))-\hat u^T(X(T,t,x))|>\gamma/2\}\right)\\
	\dsp\hphantom{\leq}+2\eta+2L_n\eta+L\eta\\
	\leq\Leb^N\left(\{x\in B_r \,:\, |X_n(T,t,x)-X(T,t,x)|>\alpha\}\right)+2\eta+2L_n\eta+L\eta\\
	\leq 3\eta+2L_n\eta+L\eta,
	\end{array}
\end{equation*}
which proves (i).

For proving (ii), because of (i) we need only to prove the equi-integrability of $|u_n(t,\cdot)|^q$
for $n\in\N$ and $t\in[0,T]$. This holds true by using Lemma~\ref{l:equiint}, the uniformly bounded compression property and Remark~\ref{r:Xsmallinfinity}.
The case $q=\infty$ is obvious since the $L^\infty$ bound ensures the local equi-integrability.
\end{proof}

If in addition to the previous conditions we further assume $\div b \in L^1 \big( (0,T) ; L^\infty(\R^N) \big)$,
we have existence and uniqueness of the backward regular Lagrangian flow starting at time $t \in [0,T]$,
according to Corollary~\ref{c:finalwell}. This gives the possibility of defining the Lagrangian solution to the Cauchy problem
for the transport equation in which we prescribe the initial data $u(0,\cdot) = u^0$ instead of the final data, via the formula
$u(t,x) = u^0 \big( X(0,t,x) \big)$.
The same results as above, regarding continuity, renormalized equations and stability
(with bound on $\div b_n$ in $L^1 \big( (0,T) ; L^\infty(\R^N) \big)$)
hold also in this case.\\

In a similar fashion we can consider the backward Cauchy problem for the continuity equation with prescribed final data, that is
\begin{equation}\label{e:contT}
\begin{cases}
\partial_t u + \div (bu) = 0\quad\mbox{in }(0,T)\times\R^N, \\
u(T,\cdot) = u^T.
\end{cases}
\end{equation}
\begin{definition}[Lagrangian solution to the continuity equation]\label{d:transportcons}
Assume that $b$ satisfies (R1), (R2), (R3), and
$\div b \in L^1 \big( (0,T) ; L^\infty(\R^N) \big)$.
If $u^T \in L^0(\R^N)$, we define the Lagrangian solution to \eqref{e:contT} by
\begin{equation}
	u(t,x) = u^T \big( X(T,t,x) \big)JX (T,t,x) \,.
	\label{eq:deftranspcons}
\end{equation}
\end{definition}
According to Theorem~\ref{t:unirlf} and Corollary~\ref{c:finalwell}, there exists
a unique forward-backward regular Lagrangian flow $X$ associated to $b$.
With Proposition~\ref{p:jacob}, the Jacobian $JX$ exists, thus \eqref{eq:deftranspcons} is well-defined.
\begin{propos}\label{p:transportcons} Assume that $b$ satisfies (R1), (R2), (R3), and
$\div b \in L^1 \big( (0,T) ; L^\infty(\R^N) \big)$.
Then the Lagrangian solution \eqref{eq:deftranspcons} satisfies
\begin{itemize}
\item[(i)] For all $u^T \in L^0(\R^N)$, we have $u\in C\big([0,T];L^0_\loc(\R^N)\big)$;
\item[(ii)] For all $u^T\in L^q(\R^N)$ for some $1\leq q<\infty$, we have
$u\in C\big([0,T];L^q(\R^N)\big)$. For all $u^T\in L^\infty(\R^N)$, we have
$u\in  C\big([0,T];L^\infty(\R^N)-w*\big)\cap C\big([0,T];L^1_\loc(\R^N)\big)$;
\item[(iii)] For all $u^T \in L^0(\R^N)$
and $\beta\in C^1(\R)$ with $\beta$, $\beta'(z)(1+|z|)$ bounded,
\begin{equation}
	\partial_t\bigl(\beta(u)\bigr)+b\cdot\nabla\bigl(\beta(u)\bigr)
	+\beta'(u)u\,\div b=0\quad
	\mbox{in }(0,T)\times\R^N,
	\label{eq:rntranspcons}
\end{equation}
where we define $b\cdot \nabla v\equiv\div (bv) - v \, \div b$.
\end{itemize}
\end{propos}
\begin{proof} For (i), the proof is similar to that of Propositions~\ref{p:transport}(i), using \eqref{eq:regJ}.

For (ii) and $q<\infty$, taking into account (i) we just have to prove that $(|u(t,\cdot)|^q)_{0\leq t\leq T}$ is equi-integrable,
which follows from  Propositions~\ref{p:transport}(ii) and the fact that $JX$ is bounded.
The case $q=\infty$ is obvious.

For (iii), we only need to prove \eqref{eq:rntranspcons} for $\beta(0)=0$, $\beta\in C_c^1(\R)$.
Then it is enough to consider the case $u^T\in L^\infty(\R^N)$ (we can replace $u^T$
by $u^T\ind_{|u^T|\leq \lambda L}$ with $\lambda$ such that the support of $\beta$ lies
in the ball $B_\lambda$), and then by approximation it is also enough to prove \eqref{eq:rntranspcons} for $u^T\in C_c^\infty(\R^N)$.
This is obtained by approximation of $b$ by a smooth sequence $b_n$ such that
$\div b_n\rightarrow\div b$ in $L^1_\loc([0,T]\times\R^N)$, $\div b_n$ bounded in $L^1 \big( (0,T) ; L^\infty(\R^N) \big)$,
which ensures that $JX_n\rightarrow JX$ in $L^1_\loc(\R^N)$, uniformly in $s,t$.
\end{proof}
\begin{propos}[Stability of the Lagrangian continuity equation]\label{p:stabtransportcons} Let $b_n$, $b$ be vector fields satisfying assumptions
(R1), (R2), (R3), and $\div b_n,\div b \in L^1 \big( (0,T) ; L^\infty(\R^N) \big)$.
Assume that $b_n\rightarrow b$ in $L^1_\loc\big([0,T] \times \R^N\big)$, and that
\begin{itemize}
\item For some decomposition $b_n / (1+|x|) = \tilde b_{n,1} + \tilde b_{n,2}$ as in assumption (R1), we have that
$$
\| \tilde b_{n,1} \|_{L^1(L^1)} + \| \tilde b_{n,2} \|_{L^1(L^\infty)}
\qquad \text{ is equi-bounded in $n$;}
$$
\item The sequence $\|\div b_n\|_{L^1 ( L^\infty )}$ is equi-bounded;
\item $\div b_n\rightarrow\div b$ in $L^1_\loc \big( [0,T]\times\R^N \big)$.
\end{itemize}
We consider the Lagrangian solutions $u_n(t,x)$ to the continuity
equation with coefficient $b_n$ and final data $u_n^T\in L^0(\R^N)$, as well
as $u(t,x)$ associated to $b$ and $u^T\in L^0(\R^N)$. Then we have the following properties:
\begin{itemize}
\item[(i)] If $u_n^T\rightarrow u^T$ in $L^0_\loc(\R^N)$,
then $u_n(t,x)\rightarrow u(t,x)$ in $C\big([0,T];L^0_\loc(\R^N)\big)$;
\item[(ii)] If $u_n^T\rightarrow u^T$ in $L^q(\R^N)$ for some $1\leq q<\infty$,
then $u_n(t,x)\rightarrow u(t,x)$ in $C\big([0,T];L^q(\R^N)\big)$.
If $u_n^T\rightarrow u^T$ in $\big(L^\infty(\R^N)-w*\big) \cap L^1_\loc(\R^N)$, then
$u_n(t,x)\rightarrow u(t,x)$ in $C\big([0,T];L^\infty(\R^N)-w*\big)\cap C\big([0,T];L^1_\loc(\R^N)\big)$.
\end{itemize}
\end{propos}
\begin{proof}Denote by $X_n$ and $X$ the respective regular Lagrangian flows of $b_n$ and $b$.
Then according to Corollary~\ref{c:finalwell} we can take for their compression constants
$L_n=\exp\|\div b_n\|_{L^1(L^\infty)}$ and $L=\exp\|\div b\|_{L^1(L^\infty)}$, which are equi-bounded.
Thus we can apply Theorem~\ref{t:stabrlf}, and $X_n$ converges to $X$ locally in measure
in $\R^N$, uniformly in $s,t$. Moreover, according to the proof of Proposition~\ref{p:jacob}
and because of the strong convergence of $\div b_n$,
we have $JX_n\rightarrow JX$ in $L^1_\loc(\R^N)$, uniformly in $s,t$, with $JX_n$ uniformly bounded.

For (i), the proof is similar to that of Proposition~\ref{p:stabtransport}(i), knowing that $JX_n\rightarrow JX$
in $L^1_\loc(\R^N)$ uniformly in $s,t$.

For proving (ii), because of (i) we need only to prove the equi-integrability of $|u_n(t,\cdot)|^q$
for $n\in\N$ and $t\in[0,T]$. This is obvious because of Proposition~\ref{p:stabtransport}(ii)
and the uniform bound on $JX_n$.
The case $q=\infty$ is obvious since the $L^\infty$ bound ensures the local equi-integrability.
\end{proof}

We can also define the Lagrangian solution to the Cauchy problem for the continuity equation
in which we prescribe the initial data $u(0,\cdot) = u^0$ instead of the final data, via the formula
$u(t,x) = u^0 \big( X(0,t,x) \big) JX (0,t,x)$.
The same results as above, regarding continuity, renormalized equations and stability, hold.

\medskip

Finally, we discuss the situation of weakly convergent initial data.

\begin{propos}[Weakly convergent initial data]\label{stabweak}
With the same assumptions as in Proposition~\ref{p:stabtransportcons}, we have the 
following properties:
\begin{itemize}
\item[(i)] If $u_n^T\rightharpoonup u^T$ weakly in $L^q(\R^N)$ for some $1\leq q<\infty$,
then $u_n(t,x)\rightharpoonup u(t,x)$ weakly in $L^q(\R^N)$, uniformly in $t\in[0,T]$.
If $u_n^T\rightharpoonup u^T$ in $L^\infty(\R^N)-w*$,
then $u_n(t,x)\rightharpoonup u(t,x)$ in $L^\infty(\R^N)-w*$, uniformly in $t\in[0,T]$;
\item[(ii)] The same holds for the transport equation.
\end{itemize}
\end{propos}
\begin{proof} Take a test function $\varphi\in L^{q'}(\R^N)$. Then for all $n$ and $t\in[0,T]$,
we can consider $\psi_n(t,y)=\varphi\big(X_n(t,T,y)\big)$. We have using \eqref{e:jacsmooth}
\begin{equation}\begin{array}{l}
	\dsp \int_{\R^N}u_n(t,x)\varphi(x)\,dx
	=\int_{\R^N}u_n^T\big(X_n(T,t,x)\big)JX_n(T,t,x)\varphi(x)\,dx\\
	\dsp\hphantom{\int_{\R^N}u_n(t,x)\varphi(x)\,dx}
	=\int_{\R^N}u_n^T\big(X_n(T,t,x)\big)JX_n(T,t,x)\psi_n\big(t,X_n(T,t,x)\big)\,dx\\
	\dsp\hphantom{\int_{\R^N}u_n(t,x)\varphi(x)\,dx}
	=\int_{\R^N}u_n^T(y)\psi_n(t,y)\,dy.
	\label{eq:dual}
	\end{array}
\end{equation}
Since $\varphi\big(X_n(s,t,y)\big)$ is the solution to the transport problem with coefficient $b_n$
with data $\varphi$ (independent of $n$) at time $s$, Proposition~\ref{p:stabtransport} yields that $\varphi \big(X_n(s,t,y)\big)\rightarrow \varphi\big(X(s,t,y)\big)$
in $L^{q'}(\R^N)$ uniformly in $s,t$ if $q'<\infty$, or in $L^\infty - w*$ and $L^1_\loc$
if $q'=\infty$. Thus we have that $\psi_n(t,y)\rightarrow \psi(t,y)=\varphi\big(X(t,T,y)\big)$
in this topology, uniformly in $t$, enabling to pass to the limit in \eqref{eq:dual} and to conclude.
The case of the transport problem works similarly, invoking Proposition~\ref{p:stabtransportcons}.
\end{proof}

\end{document}